\documentclass{article}
\usepackage{amsmath,amsfonts,amssymb,amsthm}
\usepackage{color}
\usepackage{hyperref}
\usepackage{graphicx}

\renewcommand{\leq}{\leqslant}
\renewcommand{\geq}{\geqslant}

\renewcommand{\div}{\operatorname{div}}

\newcommand{\dist}{\operatorname{dist}}
\newcommand{\Id}{\operatorname{Id}}
\newtheorem{Theorem}{Theorem}
\newtheorem{Definition}{Definition}

\newtheorem{Proposition}{Proposition}
\newtheorem{Lemma}{Lemma}

\newtheorem{Claim}{Claim}
\allowdisplaybreaks 
\begin{document}

\date{\today}
\title{On the weak uniqueness of ``viscous incompressible fluid + rigid body'' system with Navier slip-with-friction conditions in a 2D bounded domain}
\author{Marco Bravin\footnote{Institut de Math\'ematiques de Bordeaux, UMR CNRS 5251,
Universit\'e de Bordeaux, 351 cours
de la Lib\'eration, F33405 Talence Cedex, France. }}

\maketitle

\begin{abstract}
The existence of weak solutions to the ``viscous incompressible fluid + rigid body'' system with Navier slip-with-friction conditions in a 3D bounded domain has been recently proved by  G\'{e}rard-Varet and Hillairet in \cite{exi:GeH}. 
In 2D for a fluid alone (without any rigid body) it is well-known since Leray that  weak solutions are unique, continuous in time with $ L^{2} $ regularity in space and satisfy the energy equality.
In this paper we prove that these properties also hold for the 2D ``viscous incompressible fluid + rigid body'' system.
\end{abstract}

\section*{Introduction}

The problem of a rigid body immersed in an incompressible viscous fluid with different boundary conditions has been studied a lot in the past years. At a mathematical level we have a bounded domain $ \Omega $, independent in time, which is the union of two time-dependent domains $ \mathcal{F}(t) $ and $ \mathcal{S}(t) $, i.e. $ \Omega = \mathcal{F}(t) \cup \mathcal{S}(t) $ as in Figure \ref{domain}, where     
$ \mathcal{F}(t) $ is the part of the domain fulfilled by an incompressible viscous fluid, which satisfies Navier-Stokes equations and $\mathcal{S}(t) $ the part of the domain which is occupied by the body which rigidly moves following Newton's laws. The problem is to study the evolution of the motion of the fluid and of the rigid body.

Until the body does not touch the boundary, there are two separate boundaries: $\partial \Omega $ and $ \partial \mathcal{S}(t) $, where we can impose different boundary conditions.  
The most classical setting for this problem is to prescribe  no-slip boundary condition on both $ \partial \mathcal{S}(t) $ and $ \partial \Omega$. In this case a wide literature is available regarding the existence of both  weak  and  strong solutions, see \cite{GLS}, \cite{SMST}, \cite{GGH}, \cite{str:dem}. Moreover in the 2D  case   weak solutions are also continuous in time with values in $ L^2_{\sigma} $ and unique \cite{UnGS}.
Another option is to prescribe Navier slip-with-friction boundary condition on both  $ \partial \Omega $ and $ \partial \mathcal{S}(t) $. This condition naturally   appears in the rugosity limit, see \cite{BFNW}, and allows collision between the body and the boundary, see for example \cite{DGHW}, in contrast with the lack of collision in the no-slip case \cite{H}. In \cite{pla:sue}, the authors prove a first result of existence of  weak solutions in the case where $ \Omega = \mathbb{R}^3 $. In the case of a bounded domain $ \Omega $ of $\mathbb{R}^3 $ the existence of  weak solutions has been proven by G\'{e}rard-Varet and Hillairet in \cite{exi:GeH}. 
Their result can be easily adapted to the 2D case, see Theorem \ref{THM:exw} below, which is the 2D counterpart  of Theorem 1 in \cite{exi:GeH}.
Indeed  Theorem \ref{THM:exw} involves a slightly wider set of test functions for which the density property mentioned in Lemma \ref{lem:app} is guaranteed. 

In this paper we prove that the weak solutions are continuous in time with values in $ L^2_{\sigma} $ and satisfy an energy equality, see  Theorem \ref{THE:ene}.
 Finally we prove that the weak solutions are unique, which is the counterpart of \cite{UnGS} for Navier slip-with-friction boundary conditions, see  Theorem \ref{THE:uni}.

To establish the two first properties we rely on a change of variables due to \cite{IW}, see Claim \ref{cla:cov}, and some regularization processes adapted to the body motion, see  \eqref{15}-\eqref{16} (where Lemma \ref{lem:app} is used) and 
Claim \ref{regu}. 
On the other hand to establish  uniqueness we use some maximal regularity for an auxiliary system, see Theorem \ref{STR:sol} below, 
thanks to the $ \mathcal{R}$-boundedness for the Stokes operator with Robin (i.e. Navier slip-with-friction) boundary conditions presented in \cite{shimada}. Another interesting result is the work \cite{STR}, where the authors study the imaginary power of the Stokes operator with some Navier slip-with-friction boundary condition.   
Such technics are useful to extend the theory of strong solutions from Hilbert setting, for which we refer to \cite{WAN}, to $L^p-L^q $ setting, see \cite{str:dem} in the 3D case with no-slip conditions. 
Indeed the argument presented in section \ref{rel} can be implemented to prove similar existence results in both 2D and 3D for the Navier slip-with-friction boundary conditions (using a fixed point argument as in \cite{GGH}).   

Recently  the case where Navier slip-with-friction boundary conditions are prescribed only on the body boundary $\partial{ S}(t) $ and no-slip conditions on $ \partial \Omega $ has been studied:  existence of strong solutions in Hilbert spaces was proven in \cite{STR2}, existence of weak solutions was proven in \cite{exi:wea} and  a result of weak-strong uniqueness  in the 2D case is available in \cite{wea:str}.
Let us also mention \cite{J} where the author proves the small-time global controllability of the solid motion (position and velocity) in the case where $ \Omega = \mathbb{R}^2 $ and Navier slip-with-friction boundary condition are prescribed on the solid boundary. 

\begin{figure}
\centering
\includegraphics[scale = 0.75]{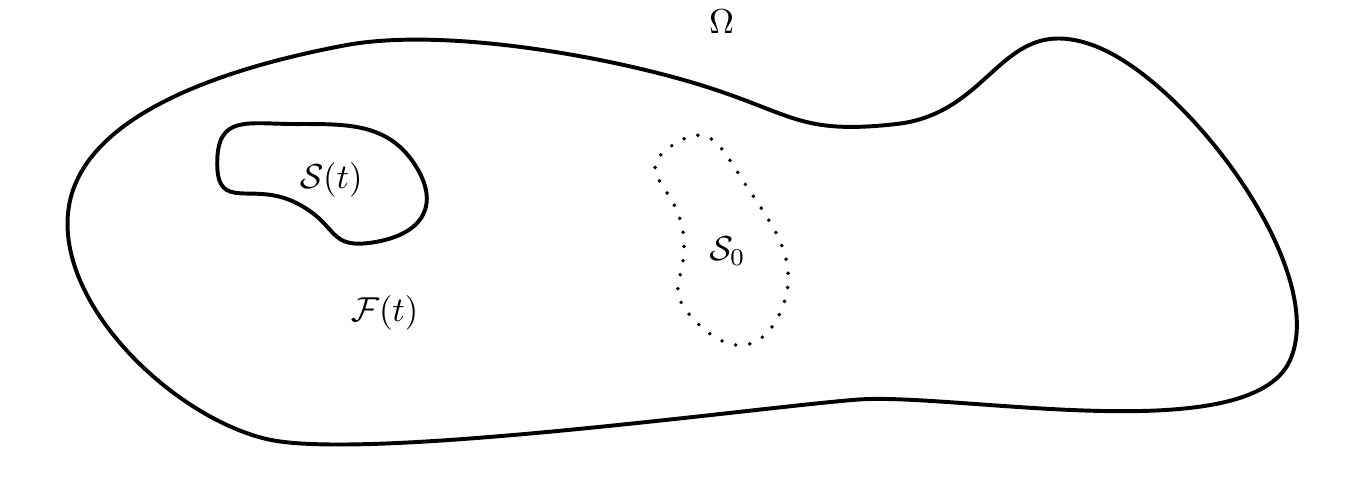}
\caption{The domain $\Omega $ is the union of two time-dependent domains $ \mathcal{F}(t)$ and $\mathcal{S}(t)$.}
\label{domain}
\end{figure}

\section{Setting}

\subsection{The ``viscous incompressible fluid + rigid body'' system}
Let us present the equations which govern the system at stake. Consider $ \Omega \subset \mathbb{R}^2 $ an open set with smooth boundary and consider $\mathcal{S}_0$ a closed, bounded, connected and simply connected subset of the plane compactly contained in $ \Omega $ with smooth boundary.
We assume that the body initially occupies the domain $\mathcal{S}_0$, has density $ \rho_{ \mathcal{S}_0}$ and rigidly moves so that at  time $t$  it occupies an isometric  domain denoted by $\mathcal{S}(t) \subset \Omega $.
We set $\mathcal{F} (t) = \Omega  \setminus \mathcal{S}(t) $ the domain occupied by the fluid  at  time $t$ starting from the initial domain $\mathcal{F}_{0}  = \Omega \setminus {\mathcal{S}}_{0} $. \par
The equations modelling the dynamics of the system  then read
\begin{eqnarray}
\displaystyle \frac{\partial u }{\partial t}+(u  \cdot\nabla)u   + \nabla p =\Delta u  && \text{for} \ x\in \mathcal{F} (t), \label{NS1-2d}\\
\div u   = 0 && \text{for} \ x\in \mathcal{F}(t) , \label{NS2-2d} \\[5pt]
u  \cdot n =   u_\mathcal{S} \cdot n && \text{for}  \  x\in \partial \mathcal{S}  (t),   \label{NS3-2d} \\
\label{NS4-2d}
(D(u) n )  \cdot \tau = - \alpha (u - u_\mathcal{S}) \cdot \tau   && \text{for} \ x\in \partial \mathcal{S}  (t),  \\[5pt]
\label{Solide1}
u\cdot n  = 0 && \text{for} \ x \in \partial \Omega, \\
(D(u) n )  \cdot \tau = - \alpha u \cdot \tau   && \text{for} \ x\in \partial \Omega, \label{Solide2} \\[5pt]
m h'' (t) &=& -  \int_{\partial  \mathcal{S} (t)} \Sigma  n \, ds ,  \label{Solide1-2d} \\
\mathcal{J} r' (t) &= & -  \int_{\partial  \mathcal{S} (t)} (x-  h (t) )^\perp \cdot \Sigma  n   \, ds , \label{Solide2-2d} \\[5pt]
u |_{t= 0} = u_0 & &  \text{for}  \  x\in  \mathcal{F}_0 ,  \label{NSci2-2d} \\
h (0)= h_0 , \ h' (0)=  \ell_0 , & &   r  (0)=  r _0.  \label{Solideci-2d}
\end{eqnarray}
Here $u=(u_1,u_2)$ and $p$ denote the velocity and pressure fields,  $n $ and $\tau$ are respectively the unit outwards normal and counterclockwise tangent vectors to the boundary of the fluid domain, $\alpha \geq 0 $ is a material constant (the friction
coefficient).
On the other hand $m$ and $ \mathcal{J}$ denote respectively the mass and the moment of inertia of the body  while the fluid  is supposed to be  homogeneous of density $1$ and the viscosity coefficient of the fluid is set equal to $1$, to simplify the notations.
The Cauchy stress tensor is defined by
$\Sigma = -p \Id_2 + 2  D(u) $, where $D(u) $ is the deformation tensor  defined by $2 D(u) = ( ( \partial_{j} u_{i} +  \partial_{i} u_{j} ) )_{1 \leqslant i,j \leqslant  2} $.
When $x=(x_1,x_2)$ the notation $x^\perp $ stands for $x^\perp =( -x_2 , x_1 )$,  $h'(t)$
is the velocity of the center of mass  $h (t)$ of the body and $r(t)$ denotes the angular velocity of the rigid body. We  denote by $u_{{\mathcal S}}$ the velocity of the body:
$u_\mathcal{S} (t,x) =   h' (t)+ r (t) (x-  h (t))^\perp $.
We assume from now on that $h_{0}= 0$.
Since $\mathcal{S}(t)$ is obtained from ${\mathcal S}_{0}$ by a rigid motion, there exists a rotation matrix
\begin{eqnarray}
\label{Q:teta}
Q (t)=
\begin{bmatrix}
\cos  \theta (t) & - \sin \theta (t) \\
\sin  \theta (t) & \cos  \theta (t)
\end{bmatrix},
\end{eqnarray}
such that  the position $\eta (t,x) \in \mathcal{S} (t)$  at the time $t$ of the point fixed to the body with an initial position $x$ is $ h (t) + Q (t)x $.
The angle $\theta$ satisfies $ \theta'(t) = r (t), $ and we choose $\theta (t)$ such that $\theta (0) =  0$.
We note also that given $ h'(t) $ and $ \theta'(t) $, we can reconstruct the position of the body trough the formula
\begin{equation*}
\mathcal{S}^{h',\theta'}(t) = \left\{ x = Q(t)y+h(t) \quad \text{ for } y \in \mathcal{S}_0 \right\},  \text{ where }
h(t) =  \int_{0}^{t} h'(t) dt \ \text{ and } \
\theta(t)=  \int_{0}^{t} \theta'(t) dt, 
\end{equation*}
and $ Q(t) $ is obtain by $ \theta $ via (\ref{Q:teta}). In the same spirit if the motion of the body is described by $ h'(t) $ and $ \theta'(t) $, then
\begin{equation*}
\rho_{\mathcal{S}(t)}(t,x) = \rho_{\mathcal{S}_0}(Q^T(t)(x-h(t))) \quad \text{ for any } x \in \mathcal{S}(t). 
\end{equation*}  

\subsection{Definition of weak solutions}
We now present the definition of weak solution and the existence result from \cite{exi:GeH}. Let $ \mathcal{O} $ be an open subset of $ \mathbb{R}^2 $ with Lipschitz boundary then we define
\begin{gather*}   
\mathcal{D}_{\sigma}(\mathcal{O}) = \{ \varphi \in \mathcal{D}(\mathcal{O}), \div \varphi = 0 \},  \quad 
L^2_{\sigma}(\mathcal{O})= \text{ the closure of } \mathcal{D}_{\sigma}(\mathcal{O}) \text{ in } L^2(\mathcal{O}), \quad H^1_{\sigma}(\mathcal{O}) = H^1(\mathcal{O}) \cap L^2_{\sigma}(\mathcal{O}) .
\end{gather*}
We also define the finite dimensional space of rigid vector fields in $ \mathbb{R}^2 $
\begin{equation}
\mathcal{R} = \{ \varphi_S, \quad \varphi_s(x) = V+\omega x^{\perp}, \quad \text{for some } V \in \mathbb{R}^2, \omega \in \mathbb{R}\},
\end{equation}
and the space of initial data
\begin{align*}
\mathcal{H}_{S_0} = \{ v \in L^2_{\sigma}(\Omega), \text{ there exists } v_{F} \in L^2_{\sigma}(\Omega)\text{, } v_S \in \mathcal{R} \text{ such that } v = v_{F} \text{ on } \mathcal{F}_0 \text{, } v = v_{S} \text{ on } \mathcal{S}_0 \}, 
\end{align*}
with norm 
\begin{equation*}
\|v\|_{\mathcal{H}_{S_0}}^2 = \int_{\mathcal{F}_0} v_F^2 dx + \int_{\mathcal{S}_0} \rho_{\mathcal{S}_0} v_S^2 dx  = \|v\|_{L^2(\mathcal{F}_0)} +m|l_v|^2+\mathcal{J}r_v^2, 
\end{equation*}
where $ l_v $ and $ r_v $ are related to $ v_S $ via $ v_S(t,x) = l_v(t) + ( x-h_0)^{\perp} r_v(t) $, with $ h_0 = 0 $ the center of mass of $ \mathcal{S}_0 $. We define for any $ T > 0 $ the space of solutions
\begin{align*}
\mathcal{V}_T = \big\{ & u \in L^{\infty}(0,T; L^{2}_{\sigma}(\Omega))\text{, there exists } u_F\in L^2(0,T;H^1_{\sigma}(\Omega))\text{, } u_S\in L^{\infty}(0,T; \mathcal{R}) \\
& \text{such that } u(t,.) = u_F(t,.) \text{ on } \mathcal{F}(t), u(t,.)= u_S(t,.) \text{ on } \mathcal{S}(t)\text{, for a.e. } t \in [0,T) \big\},
\end{align*}
Note that for any $ \varphi \in \mathcal{V}_t $ we have $ \varphi_f(t,.)\cdot n = \varphi_S(t,.)\cdot n $ on $ \partial \mathcal{S}(t)$; analogously we define  
\begin{align*}
\mathcal{W}_T = \big\{ & u \in L^{\infty}(0,T; L^{2}_{\sigma}(\Omega))\text{, there exists } u_F\in W^{1,\infty}(0,T;H^1_{\sigma}(\Omega))\text{, } u_S\in W^{1,\infty}(0,T; \mathcal{R}) \\
& \text{such that } u(t,.) = u_F(t,.) \text{ on } \mathcal{F}(t), u(t,.)= u_S(t,.) \text{ on } \mathcal{S}(t)\text{, for a.e. } t \in [0,T] \big\}.
\end{align*} 
Moreover we denote by $ \mathcal{W}_{0,T} $ the set of $ \varphi $ in $ \mathcal{W}_{T} $ such that $ \varphi \equiv 0 $ in a neighbourhood of $T$.
We are now able to give the definition of weak solution.

\begin{Definition}
\label{def:wes}
Let $ \Omega \subset \mathbb{R}^2 $ an open set with smooth boundary, $\mathcal{S}_0$ a closed, bounded, connected and simply connected subset of $ \Omega $, with smooth boundary and $ u_0 \in \mathcal{H}_{S_0}$. A weak solution of (\ref{NS1-2d})-(\ref{Solideci-2d}) on $[0,T) $, associated with the initial data $( \mathcal{S}_0,u_{S_0}) $ is a couple $ ( \mathcal{S}, u) $ satisfying 
\begin{itemize}
\item $\mathcal{S}(t) \subset \Omega $ is a bounded domain of $ \mathbb{R}^2 $ for all $ t \in [0,T) $, such that $\chi_S(t,x) = 1_{S(t)}(x) \in L^{\infty}((0,T)\times \Omega)$,
\item $ u $ belongs to the space $ \mathcal{V}_T $ where $ \mathcal{F}(t) = \Omega \setminus \overline{\mathcal{S}(t)}$ for all $ t \in [0,T) $,
\item for any $ \varphi \in \mathcal{W}_{0,T} $, it holds
\begin{gather}
-\int_0^T \int_{\mathcal{F}(t)}u_{F}\cdot \partial_t \varphi_F dx dt -\int_0^T\int_{\mathcal{S}(t)}\rho_S u_S\cdot \partial_t \varphi_S dx dt -\int_0^T\int_{\mathcal{F}(t)}  u_F \otimes u_F : \nabla \varphi_F dx dt \nonumber \\
+ 2 \int_0^T \int_{\mathcal{F}(t)} Du_F :D\varphi_F dx dt + 2 \alpha \int_0^t \int_{\partial \Omega} u_F \cdot \varphi_F ds dt + 2  \alpha \int_0^T \int_{\partial \mathcal{S}(t)} (u_F-u_S)\cdot (\varphi_F -\varphi_S) ds dt  \label{wek:for} \\
= \int_{\mathcal{F}(0)} u_{F,0}\cdot \varphi_F|_{t = 0} dx +\int_{\mathcal{S}(0)} \rho_S u_{S,0}\cdot\varphi_S|_{t=0} dx. \nonumber 
\end{gather}
In what follow we sometimes do not write explicitly the variables in which the integrations are made to shorten the notation.
\item $\mathcal{S} $ is transported by the rigid vector fields $ u_S $, i.e. for any $ \psi \in C^{\infty}_c([0,T); \mathcal{D}(\overline{\Omega})) $, it holds
\begin{equation}
\label{solid:motion}
\int_0^T\int_{\mathcal{S}(t)} \partial_t \psi + \int_0^T \int_{\mathcal{S}(t)} u_S\cdot \nabla \psi = -\int_{\mathcal{S}_0}\psi|_{t=0}.
\end{equation}
\end{itemize}
\end{Definition} 

The formal derivation of equations (\ref{wek:for})-(\ref{solid:motion}) from (\ref{NS1-2d})-(\ref{Solideci-2d}) is presented in Section 1 of \cite{exi:GeH}. Equation (\ref{solid:motion}) ensures that the solid is transported via the rigid vector field $ v_S $ and equation (\ref{wek:for}) is a weak version of the equations (\ref{NS1-2d})-(\ref{Solideci-2d}), in fact the sum of the first and the third term of (\ref{wek:for}) correspond to the convective derivative in the equation (\ref{NS1-2d}), the sum of the second and the fourth term of (\ref{wek:for}) corresponds to the pressure and the viscous term in (\ref{NS1-2d}) together with the Newton equations (\ref{Solide1-2d})-(\ref{Solide2-2d}) associated with the solid motion, the fifth and the sixth term correspond respectively to the boundary condition (\ref{Solide1})-(\ref{Solide2}) and (\ref{NS3-2d})-(\ref{NS4-2d}), and finally the last line corresponds to the initial condition (\ref{NSci2-2d})-(\ref{Solideci-2d}).

\subsection{An existence  result}
Let us conclude this section with recalling the existence result  from \cite{exi:GeH}.

\begin{Theorem}[Theorem 1 of \cite{exi:GeH}]
\label{THM:exw}
Let $ \Omega \subset \mathbb{R}^2 $ an open, bounded, connected set with smooth boundary, $\mathcal{S}_0$ a closed, bounded, connected and simply connected subset of $ \Omega $ with smooth boundary and $ u_0 \in\ \mathcal{H}_{S_0}$. There exists a weak solution $(\mathcal{S},u)$ to the problem (\ref{NS1-2d})-(\ref{Solideci-2d}) with initial data $ (\mathcal{S}_0, u_0) $ for some $ T > 0 $. Moreover either $ T = +\infty $ and $ \mathcal{S}(t) \Subset \Omega $  for any $ t \geq 0 $  or $ T < +\infty $ and it holds  $ \mathcal{S}(t) \Subset \Omega \text{ for } t \in [0,T) $ and $ \dist\left(\mathcal{S}(t),\partial \Omega\right) \to 0$ as $ t \to T^{-}$.
\end{Theorem}

The theorem  above  states that weak solutions exist up to collision, in fact by Definition \ref{def:wes} we have that the solid motion is continuous in time and the condition $ \mathcal{S}(t) \Subset \Omega $ implies that  $ \dist\left(\mathcal{S}(t),\partial \Omega\right) > 0 $, this means that the solid never touch the boundary until the final time $ t = T $, when  $ \dist\left(\mathcal{S}(T),\partial \Omega\right) =  0 $.
\bigskip

Theorem \ref{THM:exw} differs from Theorem 1 of \cite{exi:GeH} in two points. The first one is that Theorem \ref{THM:exw} deals with the 2D case whereas  Theorem 1 of \cite{exi:GeH} deals with the 3D case. Indeed this simplifies the proof. The second difference is the set of test functions used in (\ref{wek:for}), in fact in (\ref{wek:for}) we substitute  the space 
\begin{align}
\label{def-T}
\mathcal{T}_{0,T} = \{ & \varphi \in C^0_c([0,T); L^2_{\sigma}(\Omega)), \text{ there exists } \varphi_{F} \in C^{\infty}([0,T); \mathcal{D}_{\sigma}(\overline{\Omega}))\text{, }  \varphi_S \in C^{\infty}([0,T); \mathcal{R} )  \\ \nonumber &  \text{ such that } \varphi(t,.) = \varphi_{F}(t,.) \text{ on } \mathcal{F}(t)\text{, } \varphi(t,.) = \varphi_{S}(t,.) \text{ on } \mathcal{S}(t), \text{ for all }  t \in [0,T) \}. 
\end{align}
by $ \mathcal{W}_{0,T}$, but the weak solutions constructed in \cite{exi:GeH} satisfy (\ref{wek:for}) for any test function in $ \mathcal{W}_{0,T}$ in the 2D case. Moreover observe that there is no energy inequality in Definition \ref{def:wes}. Indeed in Theorem \ref{THE:ene} we are going to prove that any solution satisfies an energy equality. 
We give a sketch of the part of the proof that differs from the one in \cite{exi:GeH} in the appendix. 
%
\section{Main results}

In this section we present the two main results of this paper. The first one is that any weak solution from Definition \ref{def:wes} is continuous with values in $ L^2_{\sigma}(\Omega) $ and satisfies an energy equality.

\begin{Theorem}
\label{THE:ene}
Let $ \Omega \subset \mathbb{R}^2 $ an open, bounded set with smooth boundary, $\mathcal{S}_0$ a closed, bounded, connected and simply connected subset of $ \Omega $ with smooth boundary, $ u_0 \in \mathcal{H}_{S_0} $, and $ (\mathcal{S}, u) $ a weak solution of (\ref{NS1-2d})-(\ref{Solideci-2d}) with initial data $ (\mathcal{S}_0, u_0) $ for some $ T > 0 $. Then
\begin{equation}
\label{con:int}
u\in C^0\left([0,T); L^2_{\sigma}(\Omega)\right).
\end{equation}
Moreover, for every $ \tau \in [0,T)$, the following energy equality holds: 
\begin{align}
 \frac{1}{2}\int_{\mathcal{F}(\tau)}|u_F(\tau,.)|^2 + \frac{1}{2} \int_{\mathcal{S}(\tau)}\rho_S |u_S(\tau,.)|^2 & +2  \int_0^{\tau}\int_{\mathcal{F}(t)}|Du_F|^2 + 2  \alpha \int_0^{\tau} \int_{\partial \Omega} |u_F|^2 \nonumber \\ & + 2  \alpha \int_0^{\tau} \int_{\partial \mathcal{S}(t) } |u_F-u_S|^2 = \frac{1}{2}\int_{\mathcal{F}_0} |u_{F,0}|^2 + \frac{1}{2}\int_{\mathcal{S}_0} \rho_S|u_{0,S}|^2.  \label{energy}
\end{align}
\end{Theorem}
Note that the energy equality holds for every time, not only almost everywhere.
\bigskip

The second main result of this paper is to prove that weak solutions are actually unique.

\begin{Theorem}
\label{THE:uni}
Let $ \Omega \subset \mathbb{R}^2 $ an open set with smooth boundary, $ \mathcal{S}_0$ be a closed, bounded, connected and simply connected subset of $ \Omega $ with smooth boundary, $ u_0 \in \mathcal{H}_{S_0} $. 
Let $ (\mathcal{S},u) $  a weak solution of (\ref{NS1-2d})-(\ref{Solideci-2d}) with initial data $ (\mathcal{S}_0, u_0) $ for some $ T > 0 $. 
Then  $ (\mathcal{S},u) $ is the unique weak solution to (\ref{NS1-2d})-(\ref{Solideci-2d}) with initial data $ (\mathcal{S}_0, u_0)$ in $ [0,T) $. 
\end{Theorem}

Let us recall that weak-strong uniqueness has been recently proven  in   \cite{wea:str} in the slightly different case where  Navier slip-with-friction boundary conditions are prescribed only on the body boundary $\partial{ S}(t) $ while no-slip conditions are prescribed on the external boundary $ \partial \Omega $.  However Theorem \ref{THE:uni} deals with uniqueness of weak solutions without any regularity assumption of the initial data. 


\section{Introduction to the proof of Theorem \ref{THE:ene}}

In this section we present the main ingredients that we will use in the proof of Theorem \ref{THE:ene}. Let $ ( \mathcal{S},u )$ a weak solution. We start by introducing two spaces:
\begin{align*}
\mathcal{H}_{\tau} = \big\{ & v \in L^{2}(0,\tau; L^{2}_{\sigma}(\Omega))\text{, there exists } v_F\in L^2(0,\tau;L^2_{\sigma}(\Omega))\text{, } v_S\in L^{2}(0,\tau; \mathcal{R}) \\
& \text{such that } v(t,.) = v_F(t,.) \text{ on } \mathcal{F}(t), v(t,.)= v_S(t,.) \text{ on } \mathcal{S}(t)\text{, for a.e. } t \in [0,\tau] \big\},
\end{align*}
with norm $\|v\|_{\mathcal{H}_{\tau}} $ given by
\begin{equation*}
\|v\|_{\mathcal{H}_{\tau}}^2 = \int_0^{\tau}\|v_{F}\|_{L^2(\mathcal{F}(t))}^2 dt +m\int_0^{\tau}|l_v|^2(t)dt +\mathcal{J}\int_0^{\tau}r_v^2(t)dt, 
\end{equation*}
where $ v_S $ is decomposed into $ v_S(t) = l_v(t) +(x-h(t))^{\perp}r_v(t) $ and the space
\begin{align*}
\mathcal{E}_{\tau} = \big\{ & v \in L^{2}(0,\tau; L^{2}_{\sigma}(\Omega))\text{, there exists } v_F\in L^2(0,\tau;H^1_{\sigma}(\Omega))\text{, } v_S\in L^{2}(0,\tau; \mathcal{R}) \\
& \text{such that } v(t,.) = v_F(t,.) \text{ on } \mathcal{F}(t), v(t,.)= v_S(t,.) \text{ on } \mathcal{S}(t)\text{, for a.e. } t \in [0,\tau] \big\},
\end{align*}
with norm $\|v\|_{\mathcal{E}_{\tau}} $ given by
\begin{equation*}
\|v\|_{\mathcal{E}_{\tau}}^2 = \int_0^{\tau}\|v_{F}\|_{H^1(\mathcal{F}(t))}^2 dt +m\int_0^{\tau}|l_v|^2(t)dt +\mathcal{J}\int_0^{\tau}r_v^2(t)dt, 
\end{equation*}
We denote by $\mathcal{E}_{\tau}^{-1} $ the dual space of $\mathcal{E}_{\tau} $ and we embed $\mathcal{E}_{\tau} $ into $\mathcal{E}_{\tau}^{-1} $ through the inner product of $ \mathcal{H}_{\tau} $.

The second ingredient is the convective derivative. Let $ (\mathcal{S}, u) $ a weak solution, for any function in $f \in \mathcal{W}_{\tau} $ we define the convective derivative associated with $ u $ via
\begin{equation*}
\frac{D_u}{dt}f(t,x) = \begin{cases} \partial_t f_F(t,x) + u_F(t,x)\cdot \nabla f_F(t,x) \quad & \text{ for a.e. } \displaystyle{(t,x) \in \bigcup_{t \in [0,\tau] } \{t\} \times \mathcal{F}(t)}, \\[10pt]
\partial_t f_S(t,x) + u_S(t,x)\cdot \nabla f_S(t,x) & \text{ for a.e. } \displaystyle{(t,x) \in \bigcup_{t \in [0,\tau] } \{t\} \times \mathcal{S}(t)},
\end{cases}
\end{equation*} 
In what follows we will not write the dependence on $ u $ of the convective derivative. Moreover note that the second line of the convective derivative can be rewrite in the following way:
\begin{equation*}
\frac{D}{dt}f(t,x) = l'_f(t)+(x-h(t))^{\perp}r'_f(t)-(x-h(t))r_u(t)r_f(t) \quad  \text{ for a.e. } (t,x) \in \bigcup_{t \in [0,\tau] } \{t\} \times \mathcal{S}(t),
\end{equation*} 
where $ f_S $ is decomposed into $ f_S(t,x) = l_f(t)+(x-h(t))^{\perp}r_f(t)$.

\begin{Definition}
Given $ v \in \mathcal{V}_{\tau} $, we say that $ v $ admits a convective derivative 
\begin{equation*}
\frac{D}{dt} v \in \mathcal{E}_{\tau}^{-1}  
\end{equation*}
 if there exists a representative $ v $ and $ F \in \mathcal{E}_{\tau}^{-1} $ such that for almost every $ t_1 < t_2 \in [0,\tau ] $, it holds  
\begin{equation*}
\left\langle F , 1_{(t_1,t_2)}\varphi \right\rangle_{\mathcal{E}_{\tau}^{-1}, \mathcal{E}_{\tau}  }  = (v(t_2),\varphi(t_2))_{\mathcal{H}_{S(t_2)}} - (v(t_1),\varphi(t_1))_{\mathcal{H}_{S(t_1)}} - \int_{t_1}^{t_2}\int_{\mathcal{F}(t)} v \cdot \frac{D}{dt} \varphi dx dt - \int_{t_1}^{t_2}\int_{\mathcal{S}(t)} \rho_{S} v \cdot \frac{D}{dt} \varphi dx dt , 
\end{equation*}
for any $ \varphi \in \mathcal{W}_{\tau} $, where $1_{(t_1,t_2)} $ is the characteristic function on $(t_1,t_2) $. In this case we will denote 
\begin{equation*}
F = \frac{D}{dt} v.
\end{equation*}
\end{Definition}
Note that the above definition is an extension of the classical definition for the space $ \mathcal{E}_{\tau}$ and in what follows we will denote by $ \langle . , .  \rangle $ the pairing $ \langle., . \rangle_{\mathcal{E}_{\tau}^{-1}, \mathcal{E}_{\tau}  } $.

We conclude the section with a density lemma.

\begin{Lemma}
\label{lem:app}
The space $ \mathcal{W}_{\tau} $ is dense in $ \mathcal{E}_{\tau}$.
\end{Lemma}

\begin{proof}
To prove this lemma we construct an approximation sequence of element in $ \mathcal{W}_{\tau} $ that converge in $ \mathcal{E}_{\tau}$. We  present all the details of this construction because we  use the special property of this construction to prove Theorem \ref{THE:ene}.  
Let $ f $ an element of $ \mathcal{E}_{\tau}$, this element is not in the space $ \mathcal{W}_{\tau} $ because is not enough regular in time. 
To regularize  $ f $ in time  and preserve the rigidity of the motion  inside $ \mathcal{S}(t)$  we use a geometric change of variables that fix the position of the solid,  make a convolution in time in these variables and finally go back to the original variables. 
We start by defining the change of variables.
\begin{figure}
\centering
\includegraphics[scale = 0.75]{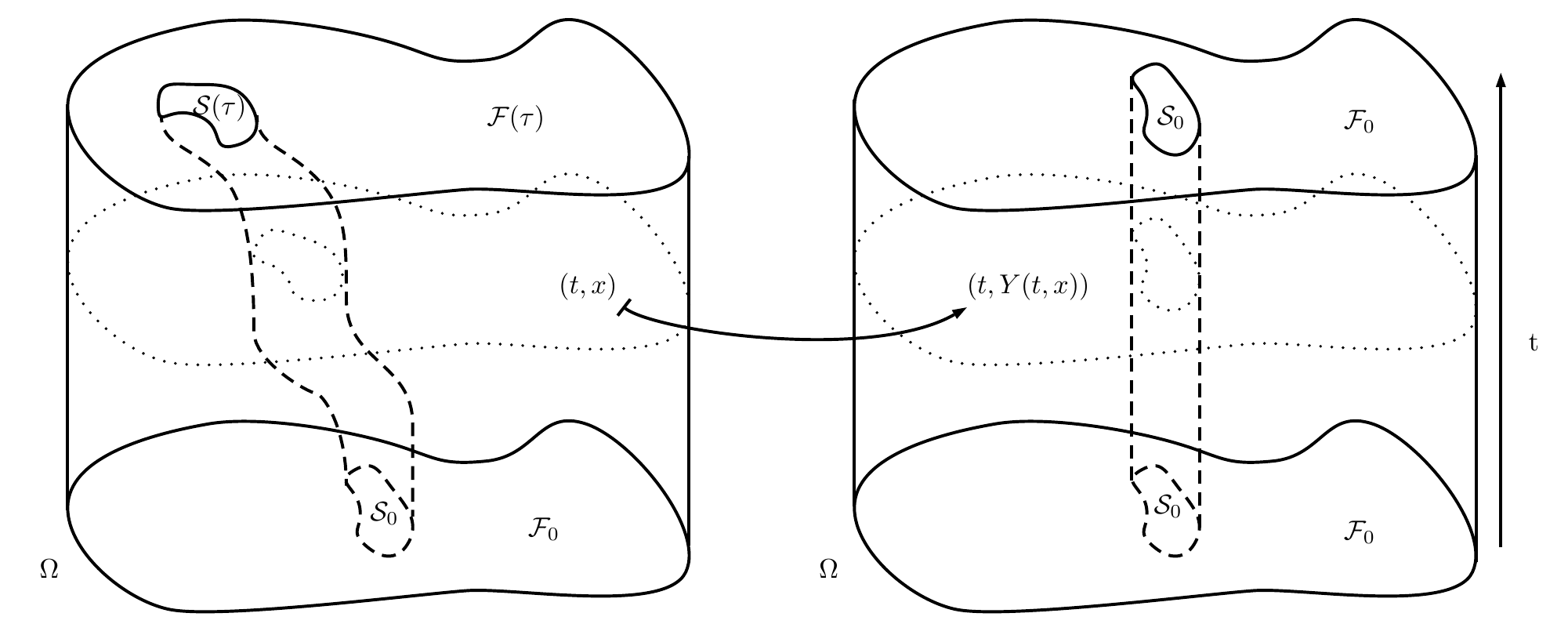}
\caption{ Change of variables $ Y_{\tau}$, which is the inverse of $ X_{\tau}$.}
\label{cha:vari}
\end{figure}
We recall from \cite[Proposition 2.1]{IW},
\cite[Lemma 2.1 and Lemma 2.2]{cha:tak}   (see also \cite[ Lemma 6.1 and Lemma 6.2]{str:dem}) the following change of variables that fixes the solid and is the identity on a neighbourhood of $ \partial \Omega $, see also Figure \ref{cha:vari}. Let $ h \in W^{1,\infty}([0,\tau], \mathbb{R}^2)$ and $ r \in L^{\infty}([0,\tau], \mathbb{R})$ associate with $ u_S $ via $ u_S(t,x) = h'(t) + (x-h(t))^{\perp}r(t)$, $ h(0) = 0 $. Let $ \alpha $ such that
%
$\min_{t \in [0,\tau] } \left(\dist(\mathcal{S}(t),\partial \Omega) \right) \geq \alpha$,
%
the existence of such an $ \alpha $ comes from the fact that by definition of weak solution the motion of $ \mathcal{S} $ is continuous in time, which implies that for any $ \tau < T $ we have
%
$\inf_{t\in[0,\tau]} \left( \dist{\mathcal{S}(t),\partial \Omega} \right) = \min_{t\in[0,\tau]} \left( \dist{\mathcal{S}(t),\partial \Omega} \right) > 0$.
%
Finally let $ \psi \in C^{\infty}_c(\Omega; \mathbb{R}) $ be a cut-off, such that $ \psi \equiv 1 $ for any $ x $ such that $\dist( x , \partial \Omega ) \geq \alpha/2 $ and $ \psi \equiv 0 $ for any $ x $ such that $\dist( x , \partial \Omega ) \leq \alpha/4 $. We define $ w $ in $ [0,\tau] \times \Omega $ by
\begin{equation*}
w(t,x) = ( x -h(t))^{\perp}\cdot h'(t)+ \frac{|x-h(t)|^2}{2}r(t),
\end{equation*}
and we define $ \Lambda : [0,\tau] \times \Omega \to \mathbb{R}^2 $ by
\begin{equation}
\label{def:Lam}
\Lambda(t,x) = \begin{pmatrix}
\displaystyle{-\frac{\partial \psi}{\partial x_2} w + \psi u_{S,1}} \\
\displaystyle{\frac{\partial \psi }{\partial x_1} w +\psi u_{S,2}}
\end{pmatrix},
\end{equation} 
where $ u_{S,i} $ is the $ i $-th component of $ u_S $.
Then $ \Lambda \in L^{\infty}(0,T; C^{k}(\Omega))$ for any $ k \in \mathbb{N} $, 
 $ \Lambda(t,x) = 0 $ for all $t$ in $ (0,\tau) $ when $ \dist(x, \partial \Omega ) < \alpha/4 $,
 $ \div \Lambda(t,x) = 0 $ for any $ (t,x) $ and 
 $ \Lambda(t,x) = h'(t) + r(t)(x-h(t))^{\perp} $ for any   $t$ in $ (0,\tau) $ and for any  $x$ in $ \mathcal{S}(t)$.

\begin{Claim}
\label{cla:cov}
Let $ \Lambda $ defined in (\ref{def:Lam}). Then there exists a unique solution $ X:[0,\tau]\times \Omega \to \Omega $ with $ X \in W^{1,\infty}(0;T; C^{k}(\Omega)) $ for any $ k \in \mathbb{N} $ of the equation 
\begin{equation*}
\begin{cases}
\partial_t X(t,y) = \Lambda(t,X(t,y)) \quad & \text{ in } (0,T)\times\Omega, \\
X(0,y) = y \quad & \text{ in } \Omega. 
\end{cases}
\end{equation*}  
Moreover it holds
\begin{itemize}
\item $ X(t,.) $ is a $ C^{\infty}$-diffeomorphism for any $ t \in [0,\tau]$,
\item $ \det \nabla X(t,.) = 1 $ for any $ t \in [0,\tau]$,
\item $ Y(t,.) = [X(t,.)]^{-1} $ is the inverse of $ X(t,.) $ for any $ t \in [0,\tau]$.
\end{itemize}
\end{Claim}

We are now able to smoothen the solution in time in the following way. Let $ \eta \in C^{\infty}_c(-1,1) $ be an even function such that $ \eta = 1 $ in a open neighbourhood of $ 0 $, $ 0 \leq \eta \leq 1 $ and $ \int \eta = 1 $ and let $ \eta_{\varepsilon} = \eta (./\varepsilon)/\varepsilon $.
Let $ \psi \in C^{\infty}_c(\mathbb{R}) $ such that $ 0 \leq \psi \leq 1 $ and such that $ \psi \equiv 1 $ in an open neighbourhood of $ [0,\tau] $. Finally let $ X_{\tau} $ be the extension in $(-\infty, +\infty ) $ of $ X $ defined in Claim \ref{cla:cov}, i.e.
\begin{equation}
\label{X:def}
X_{\tau}(t,y) = \begin{cases}  X(0,y) \quad & \text{ for } t \leq 0 \text{ and any } y \in \Omega, \\ X(t,y) \quad & \text{ for } t \in (0,\tau) \text{ and any } y \in \Omega, \\ X(\tau,y) \quad & \text{ for } t \geq \tau \text{ and any } y \in \Omega. \end{cases}
\end{equation} 
And in analogous way we extend the inverse $ Y_{\tau}$, $ h_{\tau}$ and $ Q_{\tau}$. In what follows we do not write the index $ \tau $ for simplicity.

We introduce the functions
\begin{align}
\tilde{v}(t,y) = & \nabla Y(t,X(t,y)) f(t,X(t,y)),\nonumber \\
\tilde{v}_{S}(t,y) = &  Q^T(t) f_S(t,h(t)+Q(t)y), \label{ext:v} \\ 
\tilde{v}_{F}(t,y) = & \nabla Y(t,X(t,y)) f_F(t,X(t,y)). \nonumber
\end{align}
It is clear that $ \tilde{v} \in L^{\infty}(0,\tau ;L^2_{\sigma}(\Omega)) $,  $ \tilde{v}_S \in L^{2}(0,\tau;\mathcal{R}) $ and $ \tilde{v}_F \in L^{2}(0,\tau ; H^1_{\sigma}(\Omega))$.  

Let $ v $, $ v_S $ and $ v_F $ the following extension of $ \tilde{v} $, $ \tilde{v}_S $ and $ \tilde{v}_F $ in $ (-\infty, +\infty) $, i.e.
\begin{equation*}
v(t,.) = \begin{cases} \psi(t) \tilde{v}(0,.) \quad & \text{ for } t \leq 0, \\ \psi(t) \tilde{v}(t,.) \quad & \text{ for } t \in (0,\tau) \\ \psi(t)\tilde{v}(\tau,.) \quad & \text{ for } t \geq \tau, \end{cases},
\end{equation*}  
then we define $ v_{\varepsilon} = \eta_{\varepsilon} * v $ and in an analogous way $ v_{S,\varepsilon}  = \eta_{\varepsilon}*v_{S}$ and $ v_{F,\varepsilon} = \eta_{\varepsilon}*v_{F} $. It is clear from Figure \ref{cha:vari} that when we convolute in time we average velocity associated or only with the fluid, in the case $ y \in \mathcal{F}_0 $ or only with the body, in the case $ y \in \mathcal{S}_0 $. We are now able to define
\begin{align}
f_{\varepsilon}(t,x) = \nabla X(t,Y(t,x)) v_{\varepsilon}(t,Y(t,x)), \nonumber \\
f_{S,\varepsilon}(t,x) =Q(t) v_{S,\varepsilon}(t,Q^{T}(x-h(t))), \label{u:appr} \\
f_{F,\varepsilon}(t,x) = \nabla X(t,Y(t,x)) v_{F,\varepsilon}(t,Y(t,x)). \nonumber 
\end{align}
Note that 
\begin{equation*}
f_{S,\varepsilon}(t,x) = Q(t)\left(\eta_{\varepsilon}*(Q^T l_f)\right)(t)+(x-h(t))^{\perp}\eta_{\varepsilon}*r_f(t). 
\end{equation*}
Then it is straightforward that $ f_{\varepsilon} \in \mathcal{W}_{\tau} $ (observe that $ X(t,y) = h(t)+Q(t)y $ in a neighbourhood of $ \mathcal{S}(t) $) and that
$f_{\varepsilon} \to f$ in $ L^{2}(0,\tau; L^2_{\sigma}(\Omega))$,
$f_{F,\varepsilon} \to f_F$ in $ L^2(0,\tau; H^1_{\sigma}(\Omega))$ and 
$f_{S,\varepsilon} \to f_S $  in $ L^2(0,\tau; \mathcal{R}).$
\end{proof}

\section{Proof of Theorem \ref{THE:ene}}
We start with the proof of the energy equality. Let $ ( \mathcal{S} , u ) $ a weak solution with initial data $ (\mathcal{S}_0, u_0) $ for some $ T > 0 $. Fix a representative of $u$. 
 For almost every $ \tau \in [0,T) $ it holds:
\begin{gather}
-\int_0^{\tau} \int_{\mathcal{F}(t)}u_{F}\cdot \partial_t \varphi_F dx dt -\int_0^{\tau}\int_{\mathcal{S}(t)}\rho_S u_S\cdot \partial_t \varphi_S dx dt -\int_0^{\tau}\int_{\mathcal{F}(t)}  u_F \otimes u_F : \nabla \varphi_F dx dt \nonumber \\
+ 2 \int_0^{\tau} \int_{\mathcal{F}(t)} Du_F :D\varphi_F dx dt + 2 \alpha \int_0^{\tau} \int_{\partial \Omega} u_F \cdot \varphi_F ds dt + 2  \alpha \int_0^{\tau} \int_{\partial \mathcal{S}(t)} (u_F-u_S)\cdot (\varphi_F -\varphi_S) ds dt  \label{wek:sho} \\
= \int_{\mathcal{F}(0)} u_{F,0}\cdot \varphi_F|_{t = 0} dx +\int_{\mathcal{S}(0)} \rho_S u_{S,0}\cdot\varphi_S|_{t=0} dx - \int_{\mathcal{F}(\tau)} (u_F\cdot \varphi_F)|_{t = \tau} dx - \int_{\mathcal{S}(\tau)} (\rho_S u_{S,0}\cdot\varphi_S)|_{t=\tau} dx. \nonumber 
\end{gather}
for all test functions $ \varphi \in \mathcal{W}_{\tau} $. We can obtain the energy inequality by testing the equation (\ref{wek:for}) by the solution $ u $ itself at a formal level. To do this in a rigorous way we reformulate (\ref{wek:for}) in such a way that we can test with less regular in time functions. We notice that
\begin{equation}
\label{reg:der}
\frac{D}{dt} u \in \mathcal{E}_{\tau}^{-1}.  
\end{equation}
Indeed  (\ref{wek:sho}) tells us that for almost every $ t_1 < t_2 \in [0,\tau ]$ it holds
\begin{gather*}
-\int_{t_1}^{t_2} \int_{\mathcal{F}(t)}u_{F}\cdot \partial_t \varphi_F-\int_{t_1}^{t_2}\int_{\mathcal{S}(t)}\rho_S u_S\cdot \partial_t \varphi_S 
 + \int_{\mathcal{F}({t_2})} (u_F\cdot \varphi_F)|_{t = {t_2}}  + \int_{\mathcal{S}(t_2)} (\rho_S u_{S}\cdot \varphi_S)|_{t={t_2}} \\ -
 \int_{\mathcal{F}(t_1)} u_{F}\cdot \varphi_{F}|_{t = t_1}-\int_{\mathcal{S}(t_1)} \rho_S u_{S}\cdot \varphi_{S}|_{t=t_1} - \int_{t_1}^{t_2}\int_{\mathcal{F}(t)}  u_F \otimes u_F : \nabla \varphi_{F} \\
= - 2 \int_{t_1}^{t_2} \int_{\mathcal{F}(t)} Du_F :D\varphi_{F} - 2 \alpha \int_{t_1}^{t_2} \int_{\partial \Omega} u_F \cdot \varphi_{F} - 2  \alpha \int_{t_1}^{t_2} \int_{\partial \mathcal{S}(t)} (u_F-u_S)\cdot (\varphi_{F} -\varphi_{S}) 
\end{gather*}
and the following estimate holds
\begin{align*}
\bigg| 2 \int_0^{\tau} \int_{\mathcal{F}(t)} Du_F :D\varphi_{F} +  2 \alpha \int_0^{\tau} \int_{\partial \Omega} u_F \cdot \varphi_{F} +  2  \alpha \int_0^{\tau} \int_{\partial \mathcal{S}(t)} (u_F-u_S)\cdot (\varphi_{F} -\varphi_{S}) \bigg| \leq C\|\varphi\|_{\mathcal{E}_{\tau}}.
\end{align*}
This implies that we can write the weak formulation (\ref{wek:sho}) in the following way
\begin{equation} \label{15}
\left\langle \frac{D}{dt} u, \varphi \right\rangle = - 2 \int_{0}^{\tau} \int_{\mathcal{F}(t)} Du_F :D\varphi_{F} - 2 \alpha \int_{0}^{\tau} \int_{\partial \Omega} u_F \cdot \varphi_{F} - 2  \alpha \int_{0}^{\tau} \int_{\partial \mathcal{S}(t)} (u_F-u_S)\cdot (\varphi_{F} -\varphi_{S}). 
\end{equation} 
The advantage of this formulation is that we can test it with any function in $ \mathcal{E}_{\tau}$. In fact $ \mathcal{W}_{\tau} $ is dense in $ \mathcal{E}_{\tau} $ and we can pass to the limit in norm of $ \mathcal{E}_{\tau}$. 

If we test the equation with $ u $, we obtain
\begin{equation} \label{16}
  \left\langle \frac{D}{dt} u, u \right\rangle 
= - 2 \int_0^{\tau} \int_{\mathcal{F}(t)} Du_F :Du_F - 2 \alpha \int_0^t \int_{\partial \Omega} u_F \cdot u_F - 2  \alpha \int_0^{\tau} \int_{\partial \mathcal{S}(t)} (u_F-u_S)\cdot (u_F -u_S).   
\end{equation}

For almost every $ \tau \in (0,T) $ the proof of the energy equality (\ref{energy}) therefore follow from the following claim. Finally to prove the energy equality everywhere we use the fact that that exists a continuous representative, which implies that (\ref{wek:sho}) holds for every $ \tau\in [0,T)$ so we can conclude the proof of the energy inequality.

\begin{Claim} \label{regu}
It holds
\begin{equation*}
\left\langle \frac{D}{dt} u, u \right\rangle_{\mathcal{E}_{\tau}^{-1},\mathcal{E}_{\tau}} =\frac{1}{2} \|u\|_{\mathcal{H}_{S(\tau)}}^2(\tau)- \frac{1}{2}\|u\|_{\mathcal{H}_{S(0)}}^2(0).  
\end{equation*}
\end{Claim}

\begin{proof}[Proof of the claim.]

Let $ u_{\varepsilon } $ be the approximation of $ u $ as in Lemma \ref{lem:app}, in other words let $ u_{\varepsilon} $ defined as in (\ref{u:appr}) where we replace $ f $ by $ u $. We are going to prove
\begin{equation}
\label{step:a}
\left\langle \frac{D}{dt} u_{\varepsilon}, u_{\varepsilon}\right\rangle = \frac{1}{2}\|u_{\varepsilon}\|_{\mathcal{H}_{S(\tau)}}^2(\tau)-\frac{1}{2}\|u_{\varepsilon}\|_{\mathcal{H}_{S(0)}}^2(0)
\end{equation}
and
\begin{equation}
\label{step:b}
\left\langle \frac{D}{dt} u_{\varepsilon}, u_{\varepsilon}\right\rangle = \left\langle \frac{D}{dt} u, U_{\varepsilon}\right\rangle + o(\varepsilon),
\end{equation}
where $ U_{\varepsilon} \in \mathcal{W}_{\tau} $ converges to $ u $ in $ \mathcal{E}_{\tau} $. The proof of the claim follows from (\ref{reg:der}), (\ref{step:a}) and (\ref{step:b}), in fact
\begin{equation*}
  \frac{1}{2}\|u\|_{\mathcal{H}_{S(\tau)}}^2(\tau)-\frac{1}{2}\|u\|_{\mathcal{H}_{S(0)}}^2(0)  \leftarrow  \frac{1}{2}\|u_{\varepsilon}\|_{\mathcal{H}_{S(\tau)}}^2(\tau)-\frac{1}{2}\|u_{\varepsilon}\|_{\mathcal{H}_{S(0)}}^2(0) = \left\langle \frac{D}{dt} u, U_{\varepsilon}\right\rangle + o(\varepsilon) \rightarrow \left\langle \frac{D}{dt} u, u \right\rangle,
\end{equation*}
as $\varepsilon $ goes to $ 0 $. 

To prove (\ref{step:a}), we use the fact that $ \frac{D}{dt}u_{\varepsilon} \in \mathcal{E}_{\tau}$, the identification of $ \mathcal{E}_{\tau} $ in $ \mathcal{E}_{\tau}^{-1} $ through the scalar product in $\mathcal{H}_S $ and Reynold's transport Theorem, see for instance \cite[Lemma 2.1]{wea:str}. 

Let now tackle (\ref{step:b}). We define $ U_{\varepsilon} $ as follows:
\begin{equation}
\label{def:LRUE}
l_{U_{\varepsilon}}(s) = \int_{-\infty}^{+\infty} \eta_{\varepsilon}(s-t) Q(s)Q^T(t)l_{\varepsilon}(t) dt,  \quad r_{U_{\varepsilon}}(s) = \int_{-\infty}^{+\infty}\eta(s-t)r_{\varepsilon}(t) dt, \quad U_{S,\varepsilon} = l_{U_{\varepsilon}}+(x-h(t))^{\perp}r_{U_{\varepsilon}},
\end{equation} 
where for simplicity we wrote $l_{\varepsilon}$ instead of $l_{u_\varepsilon} $ and $r_{\varepsilon}$ instead of $r_{u_\varepsilon} $ and $ u_{S, \varepsilon}(t,x) = l_{\varepsilon}(t) + (x-h(t))^{\perp}r_{\varepsilon}(t) $, 
\begin{equation*}
U_{F,\varepsilon}(t,x) = \sum_{m, l , f} \int_{-\infty}^{\infty}\eta_{\varepsilon}(s-t) \nabla Y_m(t,x) \partial_m X_l(s,Y(t,x))\partial_f X_l(s,Y(t,x)) v_{F,\varepsilon, f}(s,Y(t,x)) ds,
\end{equation*}
\begin{equation}
\label{def:UE}
U_{\varepsilon}(t,x) = \sum_{m, l , f} \int_{-\infty}^{\infty}\eta_{\varepsilon}(s-t) \nabla Y_m(t,x) \partial_m X_l(s,Y(t,x))\partial_f X_l(s,Y(t,x)) v_{\varepsilon, f}(s,Y(t,x)) ds,
\end{equation}
where $ X $ and $ Y $ are defined in (\ref{X:def}) and $ v_{\varepsilon} $ and $ v_{F,\varepsilon} $ are defined in (\ref{u:appr}), if we replace $ f_{\varepsilon} $ and $ f_{F,\varepsilon} $ by $ u_{\varepsilon} $ and $ u_{F,\varepsilon} $.

Observe that $ U_{\varepsilon} \in \mathcal{W}_{\tau} $ and $ U_{\varepsilon} $ converges to $ u $ in $\mathcal{E}_{\tau}$. To prove (\ref{step:b}), it is sufficient to prove
\begin{equation}
\label{BIGE0}
\int_0^{\tau} r'_{\varepsilon}r_{\varepsilon}  =  - \int_{0}^{\tau} r_u r'_{U_{\varepsilon}} + r_u(\tau)r_{U_{\varepsilon}}(\tau) -r_u(0)r_{U_{\varepsilon}} + o(\varepsilon),
\end{equation}
\begin{equation}
\label{BIGE1}
\int_0^{\tau}  l'_{\varepsilon}\cdot l_{\varepsilon}  = - \int_{0}^{\tau}  l_u \cdot  l'_{U_{\varepsilon}} + l_u(\tau)\cdot l_{U_{\varepsilon}}(\tau) -l_u(0)\cdot l_{U_{\varepsilon}}(0) +o(\varepsilon),
\end{equation}
and
\begin{align}
\int_{0}^{\tau} \int_{\mathcal{F}(t)} \left(\partial_t u_{\varepsilon} \cdot  u_{\varepsilon} + (u\cdot \nabla) u_{\varepsilon}\cdot u_{\varepsilon} \right)dx dt \nonumber
=  & - \int_{0}^{\tau} \int_{\mathcal{F}(t)} \left(u \cdot \partial_t U_{\varepsilon} +(u\cdot \nabla) U_{\varepsilon}\cdot u  \right)dx dt \\  & +\frac{1}{2}\int_{\mathcal{F}(\tau)} u(\tau,.) \cdot U_{\varepsilon}(\tau,.) dx - \frac{1}{2}\int_{\mathcal{F}_0} u(0,.)\cdot U_{\varepsilon}(0,.)dx +o(\varepsilon). \label{BIGE}
\end{align}

We start with the proof of (\ref{BIGE0}). From (\ref{u:appr}), we have that $ r_{\varepsilon} = \eta_{\varepsilon}*r_u $. The following computation holds
\begin{align*}
\int_0^{\tau} r'_{\varepsilon}r_{\varepsilon} = & \int_ {-\infty}^{+\infty} r'_{\varepsilon}(t)r_{\varepsilon}(t) dt -\int_{-\infty}^{0} r'_{\varepsilon}r_{\varepsilon}  -\int_{\tau}^{+\infty} r'_{\varepsilon}r_{\varepsilon}  \\
= & \int_{-\infty}^{+\infty} \left(\int_{-\infty}^{+\infty}\eta'(t-s)r_u(s)ds \right)r_{\varepsilon}(t) dt -\int_{-\infty}^{0} r'_{\varepsilon}r_{\varepsilon}  -\int_{\tau}^{+\infty} r'_{\varepsilon}r_{\varepsilon}  \\
= &- \int_{-\infty}^{+\infty} r_u(s) \left(\int_{-\infty}^{+\infty}\eta'(s-t)r_{\varepsilon}(t) dt\right) ds  -\int_{-\infty}^{0} r'_{\varepsilon}r_{\varepsilon}  -\int_{\tau}^{+\infty} r'_{\varepsilon}r_{\varepsilon}  \\
= & - \int_{0}^{\tau} r_u(s) \left(\int_{-\infty}^{+\infty}\eta'(s-t)r_{\varepsilon}(t) dt\right) ds + r_u(\tau)r_{U_{\varepsilon}}(\tau) - r_u(0)r_{U_{\varepsilon}}(0) + o(\varepsilon)\\
= & - \int_{0}^{\tau} r_u r'_{U_{\varepsilon}} + r_u(\tau)r_{U_{\varepsilon}}(\tau) -r_u(0)r_{U_{\varepsilon}} + o(\varepsilon). 
\end{align*} 
where to go from line 2 to line 3 we use the fact that $ \eta' $ is odd and in the last line we use (\ref{def:LRUE}). 

We perform similar computation to prove (\ref{BIGE1}). Clearly we have that 
\begin{equation} \label{merc1}
\int_0^{\tau}  l'_{\varepsilon}\cdot l_{\varepsilon}  = \int_{-\infty}^{+\infty}  l'_{\varepsilon}\cdot l_{\varepsilon} -\int_{-\infty}^{0}  l'_{\varepsilon}\cdot l_{\varepsilon} -\int_{\tau}^{+\infty}  l'_{\varepsilon}\cdot l_{\varepsilon}  
\end{equation}
and that 
\begin{equation}  \label{merc2}
\int_{-\infty}^{0}  l'_{\varepsilon}\cdot l_{\varepsilon} dt \to \frac{1}{2} |l_u|^2(0) \quad \text{ and } \quad \int_{\tau}^{+\infty}  l'_{\varepsilon}\cdot l_{\varepsilon} dt \to - \frac{1}{2} |l_u|^2(\tau) \quad \text{ as } \varepsilon \to 0. 
\end{equation}
We recall that by definition (\ref{u:appr}) of $ u_{S,\varepsilon}$ we have
\begin{equation*}
l_{\varepsilon}(t) = Q(t) \int_{-\infty}^{+\infty} \eta_{\varepsilon}(t-s)Q^T(s)l_u(s) ds.
\end{equation*}
Using this definition we have 
\begin{align}
\int_{-\infty}^{+\infty}  l'_{\varepsilon}(t) \cdot l_{\varepsilon}(t) dt = & \int_{-\infty}^{+\infty} Q'(t) \int_{-\infty}^{+\infty} \eta_{\varepsilon}(t-s)Q^T(s)l_u(s) ds  \cdot l_{\varepsilon}(t) dt \label{l:1} \\
& + \int_{-\infty}^{+\infty} Q(t) \int_{-\infty}^{+\infty} \eta_{\varepsilon}'(t-s)Q^T(s)l_u(s) ds \cdot l_{\varepsilon}(t) dt.  \label{l:2}
\end{align}
We use he fact that $ \eta' $ is odd and we invert the integration in $ s $ and $ t $ to arrive  at 
\begin{align}
(\ref{l:2}) =&  - \int_{-\infty}^{+\infty}  Q^T(s)l_u(s) \cdot \int_{-\infty}^{+\infty} \eta_{\varepsilon}'(s-t) Q^T(t)l_{\varepsilon}(t) dt ds \nonumber \\
= & - \int_{-\infty}^{+\infty}  Q^T(s)l_u(s) \cdot \partial_s \int_{-\infty}^{+\infty} \eta_{\varepsilon}(s-t) Q^T(t)l_{\varepsilon}(t) dt ds \nonumber
\\   = &  \int_{-\infty}^{+\infty}  l_u(s) \cdot \int_{-\infty}^{+\infty} \eta_{\varepsilon}(s-t) Q'(s) Q^T(t)l_{\varepsilon}(t) dt ds \label{l:4} \\
& - \int_{-\infty}^{+\infty}  l_u(s) \cdot \partial_s \int_{-\infty}^{+\infty} \eta_{\varepsilon}(s-t) Q(s)Q^T(t)l_{\varepsilon}(t) dt ds \label{l:5}.
\end{align}
We summarize the last computations to arrive at 
\begin{align}
 \label{merc3}
\int_{-\infty}^{+\infty}  l'_{\varepsilon}(t) \cdot l_{\varepsilon}(t) dt = (\ref{l:1}) + (\ref{l:4})+ (\ref{l:5}).
\end{align}
Moreover by the fact that $ \partial_t(Q(t)Q^T(t) ) = Q'(t)Q^T(t) +Q(t)(Q^T)'(t) = 0 $ we have
\begin{equation}  \label{merc4}
(\ref{l:1}) + (\ref{l:4}) \to 0 \quad \text{ as } \varepsilon \to 0.
\end{equation}
Gathering 
 \eqref{merc1},  \eqref{merc2}, \eqref{merc3},   \eqref{merc4} and  using that 

\begin{equation*}
-  \int_{-\infty}^{+\infty}   l_{u}\cdot l'_{U_\varepsilon}  = - \int_0^{\tau}    l_{u}\cdot l'_{U_\varepsilon}
- \frac12  l_{u} (0) \cdot l'_{U_\varepsilon}  (0) +  \frac12  l_{u}  (\tau) \cdot l'_{U_\varepsilon}  (\tau) + o(\varepsilon),
\end{equation*}
we obtain
\begin{align*}
\int_0^{\tau}  l'_{\varepsilon}\cdot l_{\varepsilon} dt = & - \int_{0}^{\tau}  l_u(s) \cdot \partial_s \int_{-\infty}^{+\infty} \eta_{\varepsilon}(s-t) Q(s)Q^T(t)l_{\varepsilon}(t) dt ds +l_u(\tau)\cdot l_{U_{\varepsilon}}(\tau) -l_u(0)\cdot l_{U_{\varepsilon}}(0) + o(\varepsilon) \nonumber \\
= & - \int_{0}^{\tau}  l_u(s) \cdot  l'_{U_{\varepsilon}} ds+l_u(\tau)\cdot l_{U_{\varepsilon}}(\tau) -l_u(0)\cdot l_{U_{\varepsilon}}(0) + o(\varepsilon).
\end{align*}
where we use (\ref{def:LRUE}).

We are left with the proof of (\ref{BIGE}). We start with the term
\begin{equation*}
\int_{0}^{\tau} \int_{\mathcal{F}(t)}\partial_t u_{\varepsilon}\cdot u_{\varepsilon} dx dt = \int_{-\infty}^{+\infty}  \int_{\mathcal{F}(t)} \partial_t u_{\varepsilon}\cdot u_{\varepsilon} dx dt -\frac{1}{2}\int_{\mathcal{F}_0}u^2(0,.) dx +\frac{1}{2}\int_{\mathcal{F}(\tau)}u^2(\tau,.) dx + o(\varepsilon).
\end{equation*}

As before we start the computation by the definition (\ref{u:appr}) of the approximate sequence $ u_{\varepsilon}$ (we exchange $ f $ with $ u $) and we compute the derivative in time. We recall the definition (\ref{u:appr}):
\begin{equation}
u_{\varepsilon, l}(t,x) = \sum_{m} \partial_m X_l(t,Y(t,x))v_{\varepsilon, m}(t,Y(t,x)).
\end{equation}
If we compute explicitly  the derivative in time we get
\begin{align}
\int_{-\infty}^{\infty}\int_{\mathcal{F}(t)}\partial_t u_{\varepsilon}\cdot u_{\varepsilon} dx dt 
= & \sum_{l,m}\int_{-\infty}^{\infty}\int_{\mathcal{F}(t)} \partial_t \partial_m X_l(t,Y(t,x))v_{\varepsilon, m}(t,Y(t,x))u_{\varepsilon, l}(t,x) dx dt           \label{I}      \\
& +  \sum_{l,m,k} \int_{-\infty}^{\infty}\int_{\mathcal{F}(t)} \partial_t Y_k(t,x)\partial_k\partial_m X_l(t,Y(t,x))v_{\varepsilon,m}(t,Y(t,x))u_{\varepsilon, l}(t,x)dxdt \label{II}\\
& + \sum_{l,m}\int_{-\infty}^{\infty}\int_{\mathcal{F}(t)} \partial_m X_l(t,Y(t,x))\partial_t v_{\varepsilon, m}(t,Y(t,x))u_{\varepsilon, l}(t,x) dx dt \label{III} \\
& +  \sum_{l,m,k} \int_{-\infty}^{\infty}\int_{\mathcal{F}(t)} \partial_m X_l(t,Y(t,x))\partial_t Y_k(t,x) \partial_k v_{\varepsilon,m}(t,Y(t,x))u_{\varepsilon,l}(t,x)dxdt. \label{IV}
\end{align} 
Using the change of variables and the fact that the determinant of the Jacobian of the change of variables is $1$ we have
\begin{align*}
(\ref{III}) =  & \sum_{l,m,f} \int_{-\infty}^{\infty}\int_{\mathcal{F}_0} \partial_m X_l(t,y)\partial_t v_{\varepsilon, m}(t,y)\partial_f X_l(t,y) v_{\varepsilon, f}(t,y) dy dt \\
= & \sum_{l,m,f} \int_{-\infty}^{\infty}\int_{\mathcal{F}_0} \int_{-\infty}^{\infty}\eta_{\varepsilon}'(t-s)v_m(s,y) ds \partial_m X_l(t,y)\partial_f X_l(t,y) v_{\varepsilon, f}(t,y) dy dt  \\
= & -\sum_{l,m,f} \int_{-\infty}^{\infty}\int_{\mathcal{F}_0} v_m(s,y) \int_{-\infty}^{\infty}\eta_{\varepsilon}'(s-t) \partial_m X_l(t,y)\partial_f X_l(t,y) v_{\varepsilon, f}(t,y) dt dy ds \\
= & -\sum_{l,m,f} \int_{-\infty}^{\infty}\int_{\mathcal{F}_0} v_m(s,y) \partial_s \left(\int_{-\infty}^{\infty}\eta_{\varepsilon}(s-t) \partial_m X_l(t,y)\partial_f X_l(t,y) v_{\varepsilon, f}(t,y) dt \right) dy ds \\
= & -\sum_{l,m,f,n} \int_{-\infty}^{\infty}\int_{\mathcal{F}_0} \partial_n Y_m(s,X(s,y))u_n(s,X(s,y)) \partial_s \left(\int_{-\infty}^{\infty}\eta_{\varepsilon}(s-t) \partial_m X_l(t,y)\partial_f X_l(t,y) v_{\varepsilon, f}(t,y) dt \right) dy ds, \\
\end{align*}
where in the second line we use the definition of convolution, in the third line we exchange the integration in $ t $ with the integration in $ s $ and we use the fact that $ \eta $ is even which implies that $ \eta' $ is odd, in the fourth line we use a property of the derivative of a convolution and in the last one we use the relation (\ref{ext:v}) between $ u $ and $ v $.

Going back to the original variables we get that the last line is equal to minus
\begin{align}
\label{1}
\sum_{l,m,f,n} \int_{-\infty}^{\infty}\int_{\mathcal{F}(s)} \partial_n Y_m(s,x)u_n(s,x) \partial_s \left(\int_{-\infty}^{\infty}\eta_{\varepsilon}(t-s) \partial_m X_l(t,Y(s,x))\partial_f X_l(t,Y(s,x)) v_{\varepsilon, f}(t,Y(s,x)) dt \right) dx ds
\end{align}
plus the following three terms
\begin{align}
\label{2}
\sum_{l,m,f,n,e} \int_{-\infty}^{\infty}\int_{\mathcal{F}(s)} \partial_n Y_m(s,x)u_n(s,x) \int_{-\infty}^{\infty}\eta_{\varepsilon}(t-s)  \partial_s Y_e(s,x) \partial_e \partial_m  X_l(t,Y(s,x)) \partial_f & X_l(t,Y(s,x)) \nonumber \\& v_{\varepsilon, f}(t,Y(s,x)) dt dx ds,
\end{align}
\begin{align}
\label{3}
\sum_{l,m,f,n,e} \int_{-\infty}^{\infty}\int_{\mathcal{F}(s)} \partial_n Y_m(s,x)u_n(s,x) \int_{-\infty}^{\infty}\eta_{\varepsilon}(t-s) \partial_m  X_l(t,Y(s,x)) \partial_s Y_e(s,x) \partial_e \partial_f & X_l(t,Y(s,x)) \nonumber \\ & v_{\varepsilon, f}(t,Y(s,x)) dt dx ds,
\end{align}
\begin{align}
\label{4}
\sum_{l,m,f,n,e} \int_{-\infty}^{\infty}\int_{\mathcal{F}(s)} \partial_n Y_m(s,x)u_n(s,x) \int_{-\infty}^{\infty}\eta_{\varepsilon}(t-s) \partial_m  X_l(t,Y(s,x)) \partial_f & X_l(t,Y(s,x)) \nonumber \\ & \partial_s Y_e(s,x) \partial_e v_{\varepsilon, f}(t,Y(s,x)) dt dx ds.
\end{align}
We isolate $ u_n(s,x) $. To do this we note that (\ref{1}) is equal to the difference of the following two terms
\begin{align}
\label{11}
\sum_{l,m,f,n} \int_{-\infty}^{\infty}\int_{\mathcal{F}(s)} u_n(s,x) \partial_s \left(\int_{-\infty}^{\infty}\eta_{\varepsilon}(t-s) \partial_n Y_m(s,x) \partial_m X_l(t,Y(s,x))\partial_f X_l(t,Y(s,x)) v_{\varepsilon, f}(t,Y(s,x)) dt \right) dx ds 
\end{align}
\begin{align}
\label{22}
\sum_{l,m,f,n} \int_{-\infty}^{\infty}\int_{\mathcal{F}(s)} u_n(s,x) \int_{-\infty}^{\infty}\eta_{\varepsilon}(t-s) \partial_s \partial_n Y_m(s,x) \partial_m X_l(t,Y(s,x))\partial_f X_l(t,Y(s,x)) v_{\varepsilon, f}(t,Y(s,x)) dt  dx ds. 
\end{align}
We arrive at
\begin{equation*}
\int_{-\infty}^{\infty}\int_{\mathcal{F}(t)}\partial_t u_{\varepsilon}\cdot u_{\varepsilon} dx dt = (\ref{I}) + (\ref{II}) - (\ref{11})+(\ref{22})+(\ref{2})+(\ref{3})+(\ref{4})+(\ref{IV}).
\end{equation*}
Notice that as $ \varepsilon $ goes to $ 0 $ we have 

\begin{gather*}
(\ref{I}) \to \sum_{l,m,n}\int_{-\infty}^{\infty}\int_{\mathcal{F}(t)} \partial_t \partial_m X_l(t,Y(t,x))\partial_n Y_m(t,x)u_n( t,x)u_l(t,x) dx dt ,\\
(\ref{22}) \to \sum_{l,m,n} \int_{-\infty}^{\infty}\int_{\mathcal{F}(s)} u_n(t,x) \partial_t \partial_n Y_m(t,x) \partial_m X_l(t,Y(t,x)) u_l(t,x) dx dt, \\
(\ref{2}) \to  
\sum_{l,m,n,e} \int_{-\infty}^{\infty}\int_{\mathcal{F}(s)} \partial_n Y_m(t,x)u_n(t,x)\partial_t Y_e(t,x) \partial_e \partial_m  X_l(t,Y(t,x))u_l(t,x) dx dt.
\end{gather*}
Moreover using that $ \partial_t(\nabla X(t,Y(t,x)\nabla Y(t,x)) = 0 $, we arrive at
\begin{equation*}
(\ref{I}) + (\ref{22}) + (\ref{2})  \to 0.
\end{equation*}

We study the terms (\ref{II}), (\ref{3}), (\ref{4}), (\ref{IV}). As $ \varepsilon $ goes to $ 0 $ we have 

\begin{gather*}
(\ref{II}),(\ref{3}) \to \sum_{l,m,k} \int_{-\infty}^{\infty}\int_{\mathcal{F}(t)} \partial_t Y_k(t,x)\partial_k\partial_m X_l(t,Y(t,x))v_m(t,Y(t,x))u_l(t,x)dxdt \\
(\ref{4}), (\ref{IV}) \to \sum_{l,m,k} \int_{-\infty}^{\infty}\int_{\mathcal{F}(t)} \partial_m X_l(t,Y(t,x))\partial_t Y_k(t,x) \partial_k v_m(t,Y(t,x))u_l(t,x)dxdt.
\end{gather*} 
Moreover it holds
\begin{align*}
\sum_{m} \partial_k\partial_m X_l(t,Y(t,x)) & v_m(t,Y(t,x)) + \partial_m X_l(t,Y(t,x))\partial_k v_m(t,Y(t,x)) \\
= & \sum_{m,i,j} \partial_k X_j(t,Y(t,x)) \partial_j Y_i(t,x) \partial_i\partial_m X_l(t,Y(t,x)) v_m(t,Y(t,x))\\ & + \sum_{m,i,j} \partial_k X_j(t,Y(t,x)) \partial_j Y_i(t,x)\partial_m X_l(t,Y(t,x))\partial_i v_m(t,Y(t,x)) \\
= & \sum_{j}  \partial_k X_j(t,Y(t,x)) \partial_j u_l(t,x),
\end{align*}
where we multiply by the identity matrix $\sum_j \partial_k X_j(t,Y(t,x)) \partial_j Y_i(t,x) = \delta_{k i} $ and by the fact that $  Y $ is the inverse of $ X $, it holds $ X(t,Y(t,x)) = x $, which implies that
\begin{equation*}
0 = \partial_t \left( X_j(t,Y(t,x)) \right) = \partial_t X_j(t,Y(t,x)) + \sum_{k} \partial_k X_j(t,Y(t,x))\partial_t Y_k(t,x).
\end{equation*}
This last two equalities lead us to prove that
\begin{align*}
\sum_{m,k} \partial_t Y_k(t,x)\partial_k\partial_m X_l(t,Y(t,x)) v_m(t,Y(t,x)) & + \partial_t Y_k(t,x)\partial_m X_l(t,Y(t,x))\partial_k v_m(t,Y(t,x)) \\
= & \sum_{j, k} \partial_t Y_k(t,x) \partial_k X_j(t,Y(t,x)) \partial_j u_l(t,x) \\
= & - \sum_{j} \partial_t X_j(t,Y(t,x))\partial_j u_l(t,x).
\end{align*}
which implies that 
\begin{align*}
(\ref{II})+(\ref{3})+(\ref{4})+(\ref{IV}) \to - 2 \sum_{j,l} \int_{-\infty}^{+\infty} \int_{\mathcal{F}(t)} \partial_t X_j(t,Y(t,x))\partial_j u_l(t,x) u_l (t,x) dx dt.
\end{align*}
We note that $ X $ (defined in (\ref{X:def})) does not change in time $ (-\infty, 0] $ and in $ [\tau,+\infty) $ and by an integration by parts we have 
\begin{align*}
(\ref{II})+(\ref{3})+(\ref{4})+(\ref{IV}) \to - 2 \int_{0}^{\tau} \int_{\mathcal{F}(t)} [(u(t,x)\cdot\nabla) u(t,x)]\cdot u(t,x) dx dt.
\end{align*}
Recall the definition of $ U_{\varepsilon}$ from (\ref{def:UE}) and let $ U_{\varepsilon, n} $ the n-th component of $ U_{\varepsilon}$, with this notation, it holds
\begin{align*}
(\ref{11}) =  \int_{-\infty}^{+\infty} \int_{\mathcal{F}(t)} u \cdot \partial_t U_{\varepsilon} dx dt = & \int_{0}^{\tau} \int_{\mathcal{F}(t)} u \cdot \partial_t U_{\varepsilon} dx dt \\ &+ \frac{1}{2}\int_{\mathcal{F}_0} u(0,.)\cdot U_{\varepsilon}(0,.)dx-\frac{1}{2}\int_{\mathcal{F}(\tau)} u(\tau,.)\cdot U_{\varepsilon}(\tau,.) dx +o(\varepsilon).
\end{align*}
To conclude the prove of (\ref{BIGE}) we note that  
\begin{align*}
\int_{0}^{\tau} & \int_{\mathcal{F}(t)} \left( \partial_t u_{\varepsilon} \cdot  u_{\varepsilon} + (u\cdot \nabla) u_{\varepsilon}\cdot u_{\varepsilon} \right) dx dt  \\
 =  &\int_{-\infty}^{+\infty} \int_{\mathcal{F}(t)} \partial_t u_{\varepsilon} \cdot  u_{\varepsilon} dx dt + \frac{1}{2}\int_{\mathcal{F}_0} |u(0,.)|^2dx-\frac{1}{2}\int_{\mathcal{F}(\tau)} |u(\tau,.)|^2dx +o(\varepsilon) + \int_{0}^{\tau} \int_{\mathcal{F}(t)} (u\cdot \nabla) u_{\varepsilon}\cdot u_{\varepsilon} dx dt  \\
= &(\ref{II})+(\ref{3})+(\ref{4})+(\ref{IV})-(\ref{11})- \int_{\mathcal{F}_0} |u(0,.)|^2dx+\int_{\mathcal{F}(\tau)} |u(\tau,.)|^2dx +o(\varepsilon) + \int_{0}^{\tau} \int_{\mathcal{F}(t)} (u\cdot \nabla) u_{\varepsilon}\cdot u_{\varepsilon} dx dt \\
= & - \int_{0}^{\tau} \int_{\mathcal{F}(t)} \left(u \cdot \partial_t U_{\varepsilon} +(u\cdot \nabla) U_{\varepsilon}\cdot u \right) dx dt  - \frac{1}{2}\int_{\mathcal{F}_0} u(0,.)\cdot U_{\varepsilon}(0,.) dx+\frac{1}{2}\int_{\mathcal{F}(\tau)} u(\tau,.)\cdot U_{\varepsilon}(\tau,.)dx +o(\varepsilon). 
\end{align*}

\end{proof}

With this last claim the proof of the energy equality is done. 
To show the continuity (\ref{con:int}) in time of the solution we follow the standard technique, but we will not present all the details because the computations are similar to the one above. The idea is to consider the approximation sequence $ u_{\varepsilon} $ defined in (\ref{u:appr}) and to prove that the sequence is a Cauchy sequence in $ C^0([0,\tau];L^2_{\sigma}(\Omega))$. To do so, we note that $ u_{\varepsilon}(0) \to u(0) $ in $ \mathcal{H}_S $, then     
\begin{align*}
\|u_{\varepsilon}(t,.)-u_{\delta}(t,.)\|_{\mathcal{H}_S} = & \|u_{\varepsilon}(0,.)-u_{\delta}(0,.)\|_{\mathcal{H}_S} + 2 \left\langle \frac{D}{dt} (u_{\varepsilon}-u_{\delta}), u_{\varepsilon}-u_{\delta} \right\rangle  \\
= & \|u_{\varepsilon}(0,.)-u_{\delta}(0,.)\|_{\mathcal{H}_S} + 2 \left\langle \frac{D}{dt} u, U_{\varepsilon,\varepsilon}-U_{\delta,\varepsilon}-U_{\varepsilon,\delta}+U_{\delta,\delta} \right\rangle  + o(\varepsilon,\delta) \\
& \to 0 \text{ as } \varepsilon, \delta \to 0,
\end{align*}
where, for $ A, B \in \{\varepsilon, \delta \}$, we set
\begin{equation*}
U_{A,B}(t,x) = \sum_{m, l , f} \int_{-\infty}^{\infty}\eta_{A}(s-t) \nabla Y_m(t,x) \partial_m X_l(s,Y(t,x))\partial_f X_l(s,Y(t,x)) v_{B, f}(s,Y(t,x)) ds,
\end{equation*}

The computation above prove that $ u_{\varepsilon} $ is a Cauchy sequence in $ C^0([0,\tau];L^2_{\sigma}(\Omega)) $, in fact $ L^2_{\sigma} $ and $ \mathcal{H}_{S(t)} $ norm are equivalent until $ \sup_{t \in[0,\tau]}\dist (\mathcal{S}(t),\partial \Omega) > 0 $, which implies that $ u \in C^0([0,\tau]; L^2_{\sigma}(\Omega))$, for any $ \tau \in [0,T) $.  

\section{Regularity in time}

Before going directly to the proof of uniqueness we present some estimates that are going to be useful in what follows. Fix now $ (\mathcal{S}, u) $ a weak solution of (\ref{NS1-2d})-(\ref{Solideci-2d}) in a time interval $[0,T)$, with $T >0$. 
We define: 
\begin{equation*}
\mathcal{F}_T = \cup_{t\in (0,T)}\{t\}\times \mathcal{F}(t).
\end{equation*}    

The first estimates are the following.

\begin{Lemma}
\label{lemma1}
The following holds true
\begin{equation*}
((u\cdot\nabla) u, u) \in L^{\frac{4}{3}}(\mathcal{F}_T; \mathbb{R}^4).
\end{equation*}
\end{Lemma}
\begin{proof}
The estimates follow by interpolation inequality and H\"older inequality.
\end{proof}

The second estimates are related to the regularization result due to viscosity. 

\begin{Lemma}
\label{reg:PROP}
There exists $ \tilde{T} \leq T $ such that the following hold true
\begin{equation*}
tu \in L^{\frac{4}{3}}((0,\tilde{T});W^{2,\frac{4}{3}}(\mathcal{F}(t))), \quad (t\partial_t u, t \nabla p)\in L^{\frac{4}{3}}(\mathcal{F}_{\tilde{T}}; \mathbb{R}^4).
\end{equation*}
\end{Lemma}
The proof of Lemma \ref{reg:PROP} is postponed to Section \ref{pro:le2}. This is the analogous result of  \cite[Proposition 3]{UnGS}.  

\section{Proof of Theorem \ref{THE:uni}}
Let $ (\mathcal{S}_1, u_1) $ and $ ( \mathcal{S}_2, u_2) $ two weak solutions of (\ref{NS1-2d})-(\ref{Solideci-2d}) on some common time interval $[0,T)$, with $T >0$. Our goal is to prove that $ (\mathcal{S}_1, u_1) = ( \mathcal{S}_2, u_2) $. We follow the same strategy than in \cite{UnGS} where the case of no-slip condition was tackled. The difficulties of the proof is due to the fact that we cannot take naively the difference of the two weak formulations and test with $ u_1-u_2 $ because the functions $ u_1 $ and $ u_2 $ are not even defined in the same domain. We use a change of variables that sends $ \mathcal{S}_2 $ to $ \mathcal{S}_1 $, to write down the weak formulation satisfied by $ \tilde{u} $ which is $ u_2 $ in this new variable, to take the difference of the two weak formulations associated with $ u_1 $ and $ \tilde{u}_2 $, to test the resulting equation  with $ u_1 -\tilde{u}_2 $ and to conclude by a Gr\"omwall estimate.

We recall that if $ (\mathcal{S}_i, u_i) $ is a weak solution, then for any $ \tau \in (0,T) $ there exist $ l_i \in C([0,\tau];\mathbb{R}^2)$ and $ r_i \in C([0,\tau];\mathbb{R}) $ such that 
$\mathcal{S}_i = \mathcal{S}^{l_i,r_i}$
and there exists  $ \delta > 0 $ such that $ \dist( \mathcal{S}_i, \partial \Omega) > \delta $ for any $ i = 1,2 $ and for any $ t \in [0,\tau] $.

We define $ X_i $ as in Claim \ref{cla:cov}, where in addition we ask that 
 $ X_i $ coincide with the solid motion associated with $\mathcal{S}_i $ for any $ (t,x) $ such that $ \dist(x, \partial \Omega) \geq \delta/2 $, 
 $ X_i $ is the identity in a $ \delta / 4 $ neighbourhood of $\delta \Omega $, i.e. $ X_i(t,x) = x $ for any $ (t,x) $ such that  $ \dist(x, \partial \Omega) \leq \delta/4 $,
and we define the change of variables $ \varphi: [0,\tau]\times\Omega \to \Omega $ and its inverse $ \psi: [0,\tau]\times\Omega \to \Omega $ as follow
$\varphi(t,x) = X_2(t,X^{-1}_1(t,x)) $ and $ \psi(t,y) = X_1(t,X_2^{-1}(t,y)).$
We easily see that $ \varphi, \psi \in C^1(0,\tau; C^{\infty}(\Omega))$. 

We can define 
\begin{align*}
\tilde{u}_2(t,x) = & \nabla \psi(t, \varphi(t, x)) u_2(t,\varphi(t,x)), \\
\tilde{u}_{F,2}(t,x) = & \nabla \psi(t, \varphi(t,x)) u_{F,2}(t, \varphi(t,x)), \\
\tilde{u}_{S,2}(t,x) = & \nabla \psi(t, \varphi(t,x)) u_{S,2}(t, \varphi(t,x)).
\end{align*}
Note that $ u_{S,2}(t,\bar{x}) = l_2(t) + (\bar{x}-h_2(t))^{\perp}r_2(t)$ then we have
\begin{align*}
\tilde{u}_{S,2}(t,x) = & \nabla \psi (t, \varphi(t,x) )\left(l_2(t) + Q_2(t)Q_1^{T}(t)(x-h_1(t))^{\perp}r_2(t) \right) \\
= & \nabla \psi (t, \varphi(t,x) )l_2(t) + (x-h_1(t))^{\perp} r_2(t). 
\end{align*}
so we define $ \tilde{l}_2(t) = Q_1(t)Q_2(t)^{T}(t)l_2(t) $ and $ \tilde{r}_2(t) = r_2(t)$, and finally by Lemma \ref{reg:PROP} in the previous section we have proved that for a short time we have improved regularity 
that leads us to define the pressure $ p_2(t,\bar{x})$, so we define $ \tilde{p}_2(t,x) = p_2(t,\varphi(t,x))$. 
We are now able to write the equation satisfied by $ \tilde{u}_2 $. We use Einstein's summation convention and we refer to \cite{UnGS} for more explicit computation.
\begin{align}
\label{long:long}
0 = & \partial_t\tilde{u}^i_2+\tilde{u}_2^j\partial_j\tilde{u}^i_2+\partial_i\tilde{p}_2-\Delta\tilde{u}^i_2 \nonumber \\
& +(\partial_k\varphi^i-\delta_{ik})\partial_t\tilde{u}^k_2+\partial_k\varphi^i\partial_l\tilde{u}^k_2(\partial_t\psi^l)+(\partial_k\partial_t\varphi^i)\tilde{u}^k_2+(\partial_{kl}^2\varphi^i)(\partial_t\psi^l)\tilde{u}^k_2  \nonumber \\
& +\tilde{u}^l_2\partial_l\tilde{u}^k_2(\partial_k\varphi^i-\delta_{ik})+(\partial_{lk}^2\varphi^i)\tilde{u}^l_2\tilde{u}^k_2+\partial_k\tilde{p}_2(\partial_i\psi^k-\delta_{ik}) \\
& -\partial_j\psi^m(\partial_{mk}^2\varphi^i)\partial_l\tilde{u}^k_2\partial_j\psi^l-(\partial_k\varphi^i\partial_j\psi^m\partial_j\psi^l-\delta_{ik}\delta_{jm}\delta_{jl})\partial_{ml}^2\tilde{u}^k_2  \nonumber  \\
& -\partial_k\varphi^i\partial_l\tilde{u}^k_2(\partial_{jj}^2\psi^l)  \nonumber  \\
& -\partial_j\psi^m(\partial_{mlk}^3\varphi^i)\partial_j\psi^l\tilde{u}^k_2-(\partial_{lk}^2\varphi^i)\partial_{jj}^2\psi^l\tilde{u}^k_2-(\partial_{lk}^2\varphi^i)\partial_j\psi^l\partial_j\psi^m\partial_m\tilde{u}^k_2. \nonumber 
\end{align}
The equation above is true almost everywhere if we restrict the time interval where the estimates of Lemma \ref{reg:PROP} hold. 
We multiply the equation above with a test function $ \varphi \in \mathcal{W}_{\tau} $ associated with the motion of $\mathcal{S}_1 $ to arrive at
\begin{gather*}
 -\int_0^{\tau} \int_{\mathcal{F}_1(t)}\tilde{u}_{F,2}\cdot \partial_t \varphi_F-\int_0^{\tau}\int_{\mathcal{S}_1(t)}\rho_S \tilde{u}_{S,2}\cdot \partial_t \varphi_S +\int_0^{\tau}\int_{\mathcal{F}_1(t)}  \tilde{u}_{F,2} \otimes \tilde{u}_{F,2} : \nabla \varphi_F  
 + 2 \int_0^{\tau} \int_{\mathcal{F}_1(t)} D\tilde{u}_{F,2} :D\varphi_F \\ + 2 \alpha \int_0^{\tau} \int_{\partial \Omega} \tilde{u}_{F,2} \cdot \varphi_F + 2  \alpha \int_0^{\tau} \int_{\partial \mathcal{S}_1(t)} (\tilde{u}_{F,2}-\tilde{u}_{S,2})\cdot (\varphi_F -\varphi_S) - \int_{\mathcal{F}_1(0)} \tilde{u}_{F,2}\cdot \varphi_F|_{t = 0}-\int_{\mathcal{S}_1(0)} \rho_S \tilde{u}_{S,2}\cdot\varphi_S|_{t=0} \\
 + \int_{\mathcal{F}_1({\tau})} (\tilde{u}_{F,2}\cdot \varphi_F)|_{t = {\tau}} + \int_{\mathcal{S}_1({\tau})} (\rho_S \tilde{u}_{S,2}\cdot\varphi_S)|_{t={\tau}} = -\int_0^{\tau} \int_{\mathcal{F}_1(t)} \tilde{f}\cdot \varphi_F dt dx ,
\end{gather*}
where $\tilde{f} $ are just the last five lines of (\ref{long:long}).
We denote by $
\hat{u} = u_1-\tilde{u}_2 $ ,
and we take the difference of the weak solution satisfies by $ u_1 $ and $ \tilde{u}_2 $ to obtain
\begin{gather*}
 -\int_0^{\tau} \int_{\mathcal{F}_1(t)}\hat{u}_{F}\cdot \partial_t \varphi_F-\int_0^{\tau}\int_{\mathcal{S}_1(t)}\rho_S \hat{u}_{S}\cdot \partial_t \varphi_S -\int_0^{\tau}\int_{\mathcal{F}_1(t)}  u_{F,1} \otimes \hat{u}_{F} : \nabla \varphi_F  
 + 2 \int_0^{\tau} \int_{\mathcal{F}_1(t)} D\hat{u}_{F} :D\varphi_F \\ + 2 \alpha \int_0^{\tau} \int_{\partial \Omega} \hat{u}_{F} \cdot \varphi_F + 2  \alpha \int_0^{\tau} \int_{\partial \mathcal{S}_1(t)} (\hat{u}_{F}-\hat{u}_{S})\cdot (\varphi_F -\varphi_S) - \int_{\mathcal{F}_1(0)} \hat{u}_{F}\cdot \varphi_F|_{t = 0}-\int_{\mathcal{S}_1(0)} \rho_S \hat{u}_{S}\cdot\varphi_S|_{t=0} \\
 + \int_{\mathcal{F}_1({\tau})} (\hat{u}_{F}\cdot \varphi_F)|_{t = {\tau}} + \int_{\mathcal{S}_1({\tau})} (\rho_S \hat{u}_{S}\cdot\varphi_S)|_{t={\tau}} = -\int_0^{\tau} \int_{\mathcal{F}_1(t)} \hat{u}_{F}\otimes u_{F,2}:\nabla \varphi_F dxdt - \int_0^{\tau} \int_{\mathcal{F}_1(t)} \tilde{f}\cdot \varphi_F dt dx. 
\end{gather*}
for any $ \varphi \in \mathcal{W}_{\tau} $.
To justify that we can  test the previous equation  with $ \hat{u} $, we follow the proof of Claim \ref{regu}, in particular we observe that $ \hat{u} \in \mathcal{U}_{\tau} $ and 
$\frac{D}{dt}u \in \mathcal{U}_{\tau}'$,
where $ \mathcal{U}_{\tau} = L^4(0,{\tau};L^4_{\sigma}(\Omega)) \cap  \mathcal{E}_{\tau} $ and 
$ \mathcal{U}_{\tau}' $ is the dual of $ \mathcal{U}_{\tau} $, where we identify $ \mathcal{U}_{\tau} $ in $ \mathcal{U}_{\tau}^{-1} $ through $ \mathcal{H}_{\tau}$.

We test with $ \hat{u} $ to obtain 
\begin{gather*}
\frac{1}{2}\int_{\mathcal{F}_1({\tau})}\hat{u}_{F}^2({\tau},.)dx +\frac{1}{2}\int_{\mathcal{S}_1({\tau})}\rho_S \hat{u}_{S}^2({\tau},.)dx + 2 \int_0^{\tau} \int_{\mathcal{F}_1(t)} D\hat{u}_{F}^2 dx dt  + 2 \alpha \int_0^{\tau} \int_{\partial \Omega} \hat{u}_{F}^2 ds dt \\ 
+ 2  \alpha \int_0^{\tau} \int_{\partial \mathcal{S}_1(t)} (\hat{u}_{F}-\hat{u}_{S})^2 ds dt 
 =  - \int_0^{\tau} \int_{\mathcal{F}_1(t)} \hat{u}_{F}\otimes u_{F,2}:\nabla \hat{u}_F dxdt  - \int_0^{\tau} \int_{\mathcal{F}_1(t)} \tilde{f}\cdot \hat{u}_F dt dx. 
\end{gather*}
We have to estimate the right hand side of the above inequality to finish the proof. The first of the two terms can be estimated via a standard technique i.e.
\begin{align*}
\left| \int_0^{\tau} \int_{\mathcal{F}_1(t)} \hat{u}_{F}\otimes u_{F,2}:\nabla \hat{u}_F dxdt   \right| \leq &  \varepsilon \int_0^{\tau} \int_{\mathcal{F}_1(t)} \nabla \hat{u}_F^2 dx dt +\frac{1}{\varepsilon}\int_0^{\tau} \left(\int_{\mathcal{F}_1(t)} \hat{u}_{F}^4 dx \right)^{1/2}\left(\int_{\mathcal{F}_1(t)} u_{F,2}^4 dx \right)^{1/2}dt \\
\leq & 2 \varepsilon \int_0^{\tau} \int_{\mathcal{F}_1(t)} \nabla \hat{u}_F^2 dx dt + \frac{C}{\varepsilon} \|u_2\|_{L^{\infty}(0,{\tau}; L^{2}(\Omega))}\int_0^{\tau} \int_{\mathcal{F}_1(t)} \hat{u}_F^2 dx \int_{\mathcal{F}_1(t)} \nabla u_{F,2}^2 dx dt. 
\end{align*}
For the second one we follow the estimate of \cite{UnGS}, in fact these estimates do not depend on the boundary condition of our problem, if we take as example the first term of $ \tilde{f} $ we have the estimates
\begin{align*}
\left|  \int_0^{\tau} \int_{\mathcal{F}_1(t)}  (\partial_k\varphi^i-\delta_{ik})\partial_t\tilde{u}^k_2 \hat{u}_{F} dx dt \right| \leq & \int_0^{\tau}   \left\| \frac{1}{t}(\partial_k\varphi^i-\delta_{ik}) \right\|_{L^{\infty}(\Omega)}\| t \partial_t\tilde{u}^k_2\|_{L^{4/3}(\mathcal{F}_1(t))} \| \hat{u}_{F} \|_{L^4(\mathcal{F}_1(t)))} dt \\
\leq & C \int_0^{\tau}   \| (\hat{l},\hat{r})\|_{L^{\infty}(0,\tau)}\| t \partial_t\tilde{u}^k_2\|_{L^{4/3}(\mathcal{F}_1(t))} \| \hat{u}_{F} \|_{L^4(\mathcal{F}_1(t)))} dt \\
\leq & C \int_0^{\tau} \| (\hat{l},\hat{r})\|_{L^{\infty}(0,\tau)}^2 \| t \partial_t\tilde{u}^k_2\|_{L^{4/3}(\mathcal{F}_1(t))}^{4/3} dt \\
& + C \int_0^{\tau} \| t \partial_t\tilde{u}^k_2\|_{L^{2/3}(\mathcal{F}_1(t))}^{2/3} \|\hat{u}\|_{L^2(\mathcal{F}_1(t))}\|\nabla \hat{u}\|_{L^2(\mathcal{F}_1(t))} dt \\
\leq & C \int_0^{\tau} \| (\hat{l},\hat{r})\|_{L^{\infty}(0,\tau)}^2 \| t \partial_t\tilde{u}^k_2\|_{L^{4/3}(\mathcal{F}_1(t))}^{4/3} dt \\
& + C \frac{1}{\varepsilon}\int_0^{\tau} \| \partial_t\tilde{u}^k_2\|_{L^{4/3}(\mathcal{F}_1(t))}^{2/3} \|\hat{u}\|_{L^2(\mathcal{F}_1(t))}^2 dt \\
& + C \varepsilon \int_0^{\tau} \|\nabla \hat{u}\|_{L^2(\mathcal{F}_1(t))}^2 dt. 
\end{align*} 
In an analogous way we can obtain the following estimates  
\begin{equation*}
\|\hat{u}(\tau,.)\|_{L^{2}(\mathcal{F}_1(t))}^2+m|l_{\hat{u}}({\tau})|^2+\mathcal{J}|r_{\hat{u}}({\tau})|^2 \leq \int_0^{\tau} C \mathcal{B}(t)\left[ \max_{s \in [0,t]}\|\hat{u}(s,.)\|_{L^{2}(\mathcal{F}_1(s))}^2+\max_{s \in [0, t]}\left|(\hat{h},\hat{\theta},\hat{l},\hat{r})\right|^2 \right] dt,
\end{equation*} 
where 
\begin{align*}
\mathcal{B}(t) = &  \|\tilde{u}_2\|_{L^{\infty}(0,T;L^{2}(\mathcal{F}_1(t))}\left(1+\|\nabla \tilde{u}_2(t,.)\|_{L^2(\mathcal{F}_1(t))}\right)+\|\tilde{u}_2\|_{L^{\infty}(0,T;L^2(\mathcal{F}_1(t)))}^{1/2}\|\nabla\tilde{u}_2(t)\|^{1/2}_{L^2(\mathcal{F}_1(t))}\|t\nabla\tilde{u}_2(t)\|_{L^4(\mathcal{F}_1(t))} \\
& +\left(\|t\partial_t\tilde{u}_2\|_{L^{4/3}(\mathcal{F}_1(t))}+\|t\tilde{u}_2\|_{W^{2,4/3}(\mathcal{F}_1(t))}+\|t\nabla\tilde{p}_2\|_{L^{4/3}(\mathcal{F}_1(t))}\right)^{4/3}.
\end{align*}
Moreover we have
\begin{equation*}
\frac{d}{dt}\left(|\hat{h}|^2 +|\hat{\theta}|^2 \right) \leq C (|\hat{l}|^2+|\hat{r}|^2+|\hat{h}|^2+|\hat{\theta}|^2)
\end{equation*}
and we have $ \mathcal{B} \in L^1(0,{\tau}) $. The Gr\"onwall lemma leads us to conclude  that  uniqueness holds locally in time. 

Moreover, by a continuation argument, we deduce that uniqueness holds on the whole time interval  $ [0,T) $ considered at the beginning of the section.

\section{Proof of Lemma \ref{reg:PROP}}{\label{pro:le2}}
\label{rel}

We go back to the proof of Lemma \ref{reg:PROP}. To do so we follow the proof of the analogous result in \cite[Proposition 3]{UnGS}. Fix $ (\mathcal{S},u) $ a weak solution of (\ref{NS1-2d})-(\ref{Solideci-2d}) and let $ l, r \in C^{0}(0,T) $ such that $ \mathcal{S} = \mathcal{S}^{l,r} $. Recall by Lemma \ref{lemma1} that
$((u\cdot\nabla) u, u) \in L^{\frac{4}{3}}(\mathcal{F}_T; \mathbb{R}^4)$.

Consider the following problem in the unknowns $ (\mathfrak{l}, \mathfrak{r}, v , p)$
\begin{eqnarray}
\displaystyle \frac{\partial v }{\partial t}  -  \Delta v + \nabla p = u-tu\cdot \nabla u  && \text{for} \ x\in \mathcal{F} (t), \label{tup1} \\
\div v   = 0 && \text{for} \ x\in \mathcal{F}(t) ,  \\
v  \cdot n =   v_\mathcal{S} \cdot n && \text{for}  \  x\in \partial \mathcal{S}  (t),    \\
(D(v) n )  \cdot \tau = - \alpha (v - v_\mathcal{S}) \cdot \tau   && \text{for} \ x\in \partial \mathcal{S}  (t),  \\
v  \cdot n =   0 && \text{for}  \  x\in \partial \Omega,    \\
(D(v) n )  \cdot \tau = - \alpha v \cdot \tau   && \text{for} \ x\in \partial \Omega ,  \\
m \mathfrak{l}' (t) &=& -  \int_{\partial  \mathcal{S} (t)} T(v,p)  n \, ds + m l(t),   \\
\mathcal{J} \mathfrak{r}' (t) &= & -  \int_{\partial  \mathcal{S} (t)} (x-  h (t) )^\perp \cdot T(v,p)  n   \, ds  +\mathcal{J} r(t),  \\
v |_{t= 0} = 0 & &  \text{for}  \  x\in  \mathcal{F}_0 ,   \\
h (0)= 0 , \ h' (0)=  0 , & &   r  (0)=  0. \label{tup2}
\end{eqnarray}
The following holds true:
\begin{enumerate}
\item Weak solution of (\ref{tup1})-(\ref{tup2}) are unique (we can test the equation with the difference of two solutions because the domain of the solutions is fixed);
\item There exists a unique strong solution $(\mathfrak{l},\mathfrak{r}, v, p)$ of (\ref{tup1})-(\ref{tup2}) in $ L^{4/3}-L^{4/3} $ for a short time;
\item Any strong solution of (\ref{tup1})-(\ref{tup2}) is a weak solution (by some integrations by parts);
\item $ tu $ is a weak solution of (\ref{tup1})-(\ref{tup2}), so it is strong. 
\end{enumerate}
This implies the regularity result of Lemma \ref{reg:PROP}. We note that the proof of point 1, 3 and 4 are exactly the same of the equivalent problems in \cite{UnGS}. It remains to prove point 2. 
We start by stating the Theorem \ref{STR:sol} that corresponds to point 2. The proof of Lemma \ref{reg:PROP} becomes a consequence of the estimates from Lemma \ref{lemma1}, point 1, 3 and 4 and Theorem \ref{STR:sol}. The idea of the proof of Theorem \ref{STR:sol} is based on \cite{str:dem} and on a fixed point argument from \cite{GGH} to conclude. 

\begin{Theorem}
\label{STR:sol}
Let $( l, r ) \in C^0(0,T) $, let $ \mathcal{S}(t) = \mathcal{S}^{l,r}(t) $, $ \mathcal{F}(t) = \mathcal{F}^{l,r}(t) $ and $ \mathcal{F}_T = \mathcal{F}^{l,r}_T $. Given $ f \in L^{4/3}(\mathcal{F}_T) $, $ g_1, g_2 \in L^{4/3}(0,T) $, then there exists a unique $L^{4/3}$-$L^{4/3}$ strong solution $ (\mathfrak{l},\mathfrak{r}, v, p) $ in $ [0,\tilde{T}] $ with $ \tilde{T} \leq T $ to the problem
\begin{eqnarray}
\displaystyle \frac{\partial v }{\partial t}  -  \Delta v + \nabla p = f  && \text{for} \ x\in \mathcal{F} (t), \label{equ:st1} \\
\div v   = 0 && \text{for} \ x\in \mathcal{F}(t) ,  \\
v  \cdot n =   v_\mathcal{S} \cdot n && \text{for}  \  x\in \partial \mathcal{S}  (t),    \\
(D(v) n )  \cdot \tau = - \alpha (v - v_\mathcal{S}) \cdot \tau   && \text{for} \ x\in \partial \mathcal{S}  (t),  \\
v  \cdot n =  0 && \text{for}  \  x\in \partial \Omega ,    \\
(D(v) n )  \cdot \tau = - \alpha v \cdot \tau   && \text{for} \ x\in \partial \Omega ,  \\
m \mathfrak{l}' (t) &=& -  \int_{\partial  \mathcal{S} (t)} T(v,p)  n \, ds +g_1,   \\
\mathcal{J} \mathfrak{r}' (t) &= & -  \int_{\partial  \mathcal{S} (t)} (x-  h (t) )^\perp \cdot T(v,p)  n   \, ds  +g_2,  \\
v |_{t= 0} = 0 & &  \text{for}  \  x\in  \mathcal{F}_0 ,   \\
h (0)= 0 , \ h' (0)=  0 , & &   r  (0)=  0. \label{equ:st2}
\end{eqnarray}
\end{Theorem}

The proof of the theorem is divided into three steps. In the first one we study the problem where the domain does not depend on time. In the second step we move the problem to a fixed domain one. In the last one we use a fixed point argument to conclude.

\subsection{Time-independent domain}  

Consider the  independent-in-time domain problem associated with (\ref{equ:st1})-(\ref{equ:st2}):

\begin{eqnarray}
\displaystyle \frac{\partial v }{\partial t}  -  \Delta v + \nabla p = f  && \text{for} \ x\in \mathcal{F}_0 , \label{lin1}\\
\div v   = 0 && \text{for} \ x\in \mathcal{F}_0 ,  \\
v  \cdot n =   v_\mathcal{S} \cdot n && \text{for}  \  x\in \partial \mathcal{S}_0,    \\
(D(v) n )  \cdot \tau = - \alpha (v - v_\mathcal{S}) \cdot \tau   && \text{for} \ x\in \partial \mathcal{S}_0,  \\
v  \cdot n =   0 && \text{for}  \  x\in \partial \Omega ,    \\
(D(v) n )  \cdot \tau = - \alpha v  \cdot \tau   && \text{for} \ x\in \partial \Omega ,  \\
m \mathfrak{l}' (t) &=& -  \int_{\partial  \mathcal{S}_0} T(v,p)  n \, ds +g_1,   \\
\mathcal{J} \mathfrak{r}' (t) &= & -  \int_{\partial  \mathcal{S}_0} (x-  h (t) )^\perp \cdot T(v,p)  n   \, ds  +g_2,  \\
v |_{t= 0} = 0 & &  \text{for}  \  x\in  \mathcal{F}_0 ,   \\
h (0)= 0 , \ h' (0)=  0 , & &   r  (0)=  0. \label{lin2} 
\end{eqnarray}

To prove existence of strong solutions we use an idea introduced by Maity and Tucsnak in \cite{str:dem} where they view the ``fluid+body system'' as a perturbation of the system of a fluid alone. We start by recalling the result on $ L^p-L^q $ regularity from Shimada in \cite{shimada}. To do so we need some notations.

Let $ \mathcal{P} $ the projection
\begin{equation*}
\mathcal{P}_{q,\mathcal{F}_0}: L^q(\mathcal{F}_0; \mathbb{R}^2) \to L^q_{\sigma}(\mathcal{F}_0).
\end{equation*}  
Then we can define the operator $ A_q: \mathcal{D}(A_q) \to L^q_{\sigma}(\mathcal{F}_0) $ such that for any $ u \in \mathcal{D}(A_q)$,
\begin{equation*}
A_q u  = \mathcal{P}_{q,\mathcal{F}_0} \Delta u,
\end{equation*}
and
\begin{equation*}
\mathcal{D}(A_q) = L^q_{\sigma}(\mathcal{F}_0)\cap\left\{ u \in W^{2,q}(\mathcal{F}_0) \text{ s.t. } D(u) \cdot \tau = -\alpha u\cdot \tau    \text{ on } \partial \mathcal{F}_0  \right\}.
\end{equation*}

\begin{Theorem}[Theorem 1.3 of \cite{shimada}] The operator $ A_q $ defined above is $ \mathcal{R}$-sectorial.
\end{Theorem}

We are going to reformulate the system (\ref{lin1})-(\ref{lin2}) in the form
\begin{equation*}
z'(t) = Az(t) + f(t), \quad z(0) = 0.
\end{equation*} 

We define $ (S(\mathfrak{l},\mathfrak{r}),S_{pr}(\mathfrak{l},\mathfrak{r})) $ to be the solution $ (v,p) $ of
\begin{eqnarray*}
-  \Delta v + \nabla p = 0  && \text{for} \ x\in \mathcal{F}_0 , \\
\div v   = 0 && \text{for} \ x\in \mathcal{F}_0 ,  \\
v  \cdot n =   (\mathfrak{l}+\mathfrak{r}x^{\perp}) \cdot n && \text{for}  \  x\in \partial \mathcal{S}_0,    \\
(D(v) n )  \cdot \tau = - \alpha (v - (\mathfrak{l}+\mathfrak{r}x^{\perp})) \cdot \tau   && \text{for} \ x\in \partial \mathcal{S}_0,  \\
  v  \cdot n =   0 && \text{for}  \  x\in \partial \Omega, \\
(D(v) n )  \cdot \tau = - \alpha v \cdot \tau   && \text{for} \ x\in \partial \Omega.
\end{eqnarray*}

\begin{Proposition}
The following estimates hold
\begin{equation*}
S(\mathfrak{l},\mathfrak{r}) \in \mathcal{L}(\mathbb{R}^3, W^{2,q}(\mathcal{F}_0)), \quad S_{pr}(\mathfrak{l},\mathfrak{r}) \in \mathcal{L}(\mathbb{R}^3, W^{1,q}_m(\mathcal{F}_0)),
\end{equation*}
where $ W^{1,q}_m(\mathcal{F}_0) = W^{1,q}(\mathcal{F}_0) \cap \left\{ f \in L^{q}(\mathcal{F}_0) \text{ such that } \int f = 0 \right\}$.
\end{Proposition}

\begin{proof} We recall that the Kirchoff potentials $ \phi_i $ with $ i = 1,2,3 $ are the solutions of the problems
\begin{equation*}
\begin{cases}
-\Delta \phi_i = 0 & \quad \text{ on } \mathcal{F}_0, \\
\nabla \phi_i\cdot n = K_i & \quad \text{ in } \partial\mathcal{S}_0, \\
\nabla \phi_i \cdot n = 0 & \quad \text{ in } \partial \Omega,
\end{cases}
\end{equation*}
where $ \phi_i : \mathcal{F}_0 \to \mathbb{R} $ and $ K_i = e_i \cdot n $ for $ i =1,2 $ and $ K_3 = x^{\perp}\cdot n $. Consider $ \tilde{v} = v - \mathfrak{l}_1\nabla \phi_1 -\mathfrak{l}_2\nabla \phi_2-\mathfrak{r}\nabla \phi_3 $. The couple $ (\tilde{v},p) $ satisfies the system
\begin{eqnarray}
-  \Delta \tilde{v} + \nabla p = 0  && \text{for} \ x\in \mathcal{F}_0 , \nonumber  \\
\div \tilde{v}   = 0 && \text{for} \ x\in \mathcal{F}_0 ,  \nonumber \\
\tilde{v}  \cdot n =   0 && \text{for}  \  x\in \partial \mathcal{S}_0,   \label{SY:SS} \\
(D(\tilde{v}) n )  \cdot \tau = - \alpha \tilde{v} \cdot \tau +h_1 \cdot \tau   && \text{for} \ x\in \partial \mathcal{S}_0,  \nonumber  \\
  \tilde{v}  \cdot n =   0 && \text{for}  \  x\in \partial \Omega, \nonumber  \\
(D(\tilde{v}) n )  \cdot \tau = - \alpha \tilde{v} \cdot \tau +h_2\cdot \tau   && \text{for} \ x\in \partial \Omega. \nonumber 
\end{eqnarray} 
with $ h_1 = \mathfrak{l}_1(D\nabla\phi_1\cdot n+ e_1- \nabla \phi_1)+\mathfrak{l}_2(D\nabla\phi_2\cdot n+ e_2- \nabla \phi_2)+\mathfrak{r}(D\nabla\phi_3\cdot n+ x^{\perp}\cdot n - \nabla \phi_3)) $ and $ h_1 = (\mathfrak{l}_1(D\nabla\phi_1\cdot n- \nabla \phi_1)+\mathfrak{l}_2(D\nabla\phi_2\cdot n- \nabla \phi_2)+\mathfrak{r}(D\nabla\phi_3\cdot n - \nabla \phi_3)) $. 

Recall that Shimada in \cite{shimada} prove $\mathcal{R}$-sectoriality for the operator associated with the system (\ref{SY:SS}) where $( h_1\cdot \tau)\cdot \tau $ and $( h_2\cdot \tau)\cdot \tau $ are the trace of a function $ h \in W^{1,q}(\mathcal{F}_0)$. To conclude the proof is enough, by Fredholm alternative, to prove uniqueness for system (\ref{SY:SS}). Uniqueness is clear by standard energy estimates.  

The linearity of $ S $ and $ S_{pr} $ is a direct consequence of the linearity of the system that they solve.

%
\end{proof}

We define the operator $ \mathcal{A}_{FS} : \mathcal{D}(\mathcal{A}_{FS}) \to \mathcal{X}  $ with 
\begin{equation*}
\mathcal{X} = L^q_{\sigma} ( \mathcal{F}_0) \times \mathbb{R}^2 \times \mathbb{R}, \quad
\mathcal{D}(\mathcal{A}_{FS}) = \left\{ (\mathcal{P}u, \mathfrak{l} , \mathfrak{r}) \in X \,  | \, \mathcal{P}u-\mathcal{P}S_{\mathcal{S}}(\mathfrak{l}, \mathfrak{r}) \in \mathcal{D}(A_q) \right\},
\end{equation*}
and
\begin{equation*}
\mathcal{A}_{FS} = \begin{pmatrix} A_q & - A_q \mathcal{P}S \\  \mathbb{K}^{-1}\mathcal{C}_1 & \mathbb{K}^{-1}\mathcal{C}_2    \end{pmatrix},
\end{equation*}
where $ \mathbb{K} $ is the mass plus the added mass matrix,
\begin{equation*}
\mathcal{C}_1(\mathcal{P}u) = \begin{pmatrix}
\displaystyle{-2 \int_{\partial \mathcal{S}_0} D(\mathcal{P}u) n d\gamma + \int_{\partial \mathcal{S}_0} N(\Delta\mathcal{P}u \cdot n)n d\gamma} \\ 
\displaystyle{-2 \int_{\partial \mathcal{S}_0} y^{\perp}\cdot D(\mathcal{P}u) n d\gamma + \int_{\partial \mathcal{S}_0} y^{\perp}\cdot N(\Delta\mathcal{P} u \cdot n)n d\gamma} 
\end{pmatrix},
\end{equation*}
\begin{equation*}
\mathcal{C}_2\begin{pmatrix} \mathfrak{l} \\ \mathfrak{r} \end{pmatrix} = \begin{pmatrix} 
\displaystyle{-2 \int_{\partial \mathcal{S}_0} D((\Id -\mathcal{P})S(\mathfrak{l}, \mathfrak{r}))n d\gamma} \\
\displaystyle{-2 \int_{\partial \mathcal{S}_0} y^{\perp}\cdot D((\Id -\mathcal{P})S(\mathfrak{l}, \mathfrak{r}))n d\gamma},
\end{pmatrix} 
\end{equation*}
where $ Nh = \phi $ is defined by $ \Delta \phi = 0 $ in $ \mathcal{F}_0$, $ \frac{\partial \phi}{\partial n} = h $ on $ \partial \mathcal{F}_0$.

Moreover we define $ N_S(h) = N(\mathbf{1}_{\mathcal{S}_0}h )$ and $ (\tilde{g}_1,\tilde{g}_2)^{T} = (\mathbb{K}^{-1}(g_1,g_2)^T)^T $.

\begin{Theorem}
$ (u,p,\mathfrak{l},\mathfrak{r}) $ is a smooth solution to the system (\ref{lin1})-(\ref{lin2}) if and only if it satisfies 
\begin{gather*}
\partial_t \begin{pmatrix} \mathcal{P} u   \\  \mathfrak{l}  \\  \mathfrak{r} \end{pmatrix} = \mathcal{A}_{FS}\begin{pmatrix} \mathcal{P} u   \\  \mathfrak{l}  \\  \mathfrak{r} \end{pmatrix} +  \begin{pmatrix} \mathcal{P}f \\ \tilde{g}_1 \\ \tilde{g}_2 \end{pmatrix}, \quad \begin{pmatrix} \mathcal{P}u(0) \\ \mathfrak{l}(0) \\\mathfrak{r}(0) \end{pmatrix} = \begin{pmatrix}
\mathcal{P} u_0 \\ \mathfrak{l}_0 \\ \mathfrak{r}_0  
\end{pmatrix}, \\
(I-\mathcal{P})u = (I-\mathcal{P})S(\mathfrak{l},\mathfrak{r}),\\
p = N( \Delta \mathcal{P} u \cdot n)-\lambda N_S((\mathfrak{l}+x^{\perp}\mathfrak{r})\cdot n).
\end{gather*}
\end{Theorem}

\begin{proof}
The proof is contained in \cite[Section 3.1]{str:dem}, in fact the only boundary condition that they use is
$ u\cdot n = (\mathfrak{l}+\mathfrak{r} x^{\perp})\cdot n $, and the second one is not relevant.
\end{proof}

\subsection{R-boundedness}
To prove $ L^p-L^q $ regularity we prove $ \mathcal{R}$-boundedness of the resolvent operator $ \mathcal{A}_{FS}$.  
\begin{Theorem}
Let $ 1 < q < \infty  $. Then $\mathcal{A}_{FS} = \mathcal{A}_{FS}^q $ is a $ \mathcal{R}$-bounded operator.
\end{Theorem}

\begin{proof}
To prove this theorem we just show that in some sense the operator $ \mathcal{A}_{FS} $ is a small perturbation of the operator $ A_q $. To do so we write
\begin{equation*}
\mathcal{A}_{FS} = \tilde{\mathcal{A}}_{FS} + \mathcal{B}_{FS} 
\end{equation*}
where
\begin{equation*}
\tilde{\mathcal{A}}_{FS} = \begin{pmatrix} A_q & -A_q\mathcal{P}S \\ 0 & 0 \end{pmatrix}, \quad B_{FS}\begin{pmatrix} 0 & 0 \\ \mathbb{K}^{-1}\mathcal{C}_1 &  \mathbb{K}^{-1}\mathcal{C}_2 \end{pmatrix}.
\end{equation*}
$ \tilde{\mathcal{A}}_{FS} $ is an $ \mathcal{R} $-bounded operator on the same domain of $ \mathcal{A}_{FS} $, in fact
\begin{equation*}
\lambda(\lambda \Id -\tilde{\mathcal{A}}_{FS})^{-1} = \begin{pmatrix} \lambda(\lambda \Id -A_q)^{-1} &  -\lambda(\lambda \Id -A_q)^{-1} \mathcal{P}S +\mathcal{P}S \\
0 & \Id 
\end{pmatrix}
\end{equation*}
and the desired resolvent estimates follow by the $ \mathcal{R} $-boundedness of $ A_q $ and the continuity of $ \mathcal{P}S $. Finally $ \mathcal{C}_1 $ and $ \mathcal{C}_2 $ are linear and contiuous operators with finite dimention codomain. The proof is exactly the same of \cite[Theorem 3.11]{str:dem}, in fact the estimates are only based on the normal boundary condition and on the interior regularity (i.e. the fact that $ u \in W^{2,q}(\mathcal{F}_0) $ or the fact that $ \div u = 0 $). This prove that $ \mathcal{B}_{FS}$ is a finite rank operator on $ \mathcal{D}(\mathcal{A}_{FS}) $, which implies that $ \mathcal{B}_{FS}$ is a $ \tilde{\mathcal{A}}_{FS}$-bounded operator with bound $ 0 $. The prove is conclude.
\end{proof}

\subsection{Change of variables}

We translate the problem (\ref{equ:st1})-(\ref{equ:st2}) to an equivalent one on a  domain fixed in time. Let $ X $ the geometric change of variables associated with $ \mathcal{S} $ define in (\ref{X:def}), following the idea of \cite[Section 3]{GGH} we define
\begin{align*}
\tilde{v}(t,y) \quad = & \quad   \nabla Y(t,X(t,y))v(t, X(t,y)), \\
\tilde{p}(t,y) \quad  = & \quad p(t,X(t,y)), \\
\tilde{\mathfrak{r}}(t) \quad = & \quad \mathfrak{r}(t), \\
\tilde{\mathfrak{l}}(t) \quad = & \quad Q^T(t)\mathfrak{l}(t), \\
\mathcal{T}(\tilde{v}(t,y),\tilde{p}(t,y)) \quad = & \quad Q^T(t)T(Q(t)u(t,y),p(t,y)Q(t)).
\end{align*}
In this new variables the equations become

\begin{eqnarray*}
\displaystyle \frac{\partial \tilde{v} }{\partial t}  + ( \mathcal{M}-\mathcal{L})\tilde{v} +\mathcal{G}\tilde{p} = \tilde{f} && \text{for} \ x\in \mathcal{F}_0 , \\
\div \tilde{v}   = 0 && \text{for} \ x\in \mathcal{F}_0 ,  \\
\tilde{v}  \cdot n =   \tilde{v}_\mathcal{S} \cdot n && \text{for}  \  x\in \partial \mathcal{S}_0,    \\
(D(\tilde{v}) n )  \cdot \tau = - \alpha (\tilde{v} - \tilde{v}_\mathcal{S}) \cdot \tau   && \text{for} \ x\in \partial \mathcal{S}_0,  \\
\tilde{v}  \cdot n =  0 && \text{for}  \  x \in \partial \Omega,    \\
(D(\tilde{v}) n )  \cdot \tau = - \alpha \tilde{v}  \cdot \tau   && \text{for} \ x\in \partial \Omega,  \\
m \tilde{\mathfrak{l}}' (t) &=& -  \int_{\partial  \mathcal{S}_0} \mathcal{T}(\tilde{v},\tilde{p})  n \, ds - m r\tilde{\mathfrak{l}}^{\perp} + \tilde{g}_1,   \\
\mathcal{J} \tilde{\mathfrak{r}}' (t) &= & -  \int_{\partial  \mathcal{S}_0} y^\perp \cdot \mathcal{T}(\tilde{v},\tilde{p})  n   \, ds  +g_2,  \\
\tilde{v} |_{t= 0} = 0 & &  \text{for}  \  x\in  \mathcal{F}_0 ,   \\
 \ \tilde{\mathfrak{l}}' (0)=  0 , & &   \tilde{\mathfrak{r}}  (0)=  0. 
\end{eqnarray*}
where $ \mathcal{M} , \mathcal{L}, \mathcal{G}, \tilde{f} $ and $ \tilde{g}_1 $ are defined in \cite{GGH} as follows
\begin{equation*}
(\mathcal{L}u)_i =  \sum_{j,k}^2 \partial_j(g^{ij}\partial_k u_i) + 2 \sum_{j,k,l = 1}^2 g^{kl}\gamma_{jk}^i \partial_{l}u_j+\sum_{j,k,l}^2\left(\partial_k(g^{kl}\Gamma_{jl}^i +\sum_{m=1}^2 g^{kl}\Gamma_{jl}^m\Gamma_{km}^i) \right)u_j,
\end{equation*}
\begin{equation*}
(\mathcal{M}u)_i = \sum_{j=1}^2 \dot{Y}_j \partial_j u_i + \sum_{j,k = 1}^2 \left(\Gamma_{jk}^i\dot{Y}_k+(\partial_kY_i)(\partial_j\dot{X}_k)\right)u_j, \quad (\mathcal{G}p)_i = \sum_{j=1}^2 g^{ij}\partial_j p, 
\end{equation*}
and
\begin{equation*}
(\tilde{f})_i(t,y) = \sum_{l = 1 }^2  \partial_l Y_i(t, X(t,y)) f_l(t,X(t,y)) \quad \text{ and } \quad  \tilde{g}_1(t) = Q^T(t) g_1(t), 
\end{equation*}
where
\begin{equation*}
g^{ij} = \sum_{k = 1}^2 (\partial_k Y_i)(\partial_k Y_j), \quad g_{ij} = \sum_{k=1}^2 (\partial_i X_k)(\partial_j X_k), \quad \Gamma_{jk}^i = \frac{1}{2}\sum_{l=1}^2 g^{jk}(\partial_i g_{jl} + \partial_j g_{lk} - \partial_k g_{ij}).
\end{equation*}

\subsection{Fixed point}

We use a fix point argument to conclude. We rewrite the system above in the form 
\begin{eqnarray*}
\displaystyle \frac{\partial \tilde{v} }{\partial t}  - \Delta \tilde{v} +\nabla \tilde{p} = F_{0}(\tilde{v}, \tilde{p}, \tilde{\mathfrak{l}}, \tilde{\mathfrak{r}})+\tilde{f} && \text{for} \ x\in \mathcal{F}_0 , \\
\div \tilde{v}   = 0 && \text{for} \ x\in \mathcal{F}_0 ,  \\
\tilde{v}  \cdot n =   \tilde{v}_\mathcal{S} \cdot n && \text{for}  \  x\in \partial \mathcal{S}_0,    \\
(D(\tilde{v}) n )  \cdot \tau = - \alpha (\tilde{v} - \tilde{v}_\mathcal{S}) \cdot \tau   && \text{for} \ x\in \partial \mathcal{S}_0,  \\
\tilde{v}  \cdot n =  0 && \text{for}  \  x \in \partial \Omega,    \\
(D(\tilde{v}) n )  \cdot \tau = - \alpha \tilde{v}  \cdot \tau   && \text{for} \ x\in \partial \Omega,  \\
m \tilde{\mathfrak{l}}' (t) +  \int_{\partial  \mathcal{S}_0} T(\tilde{v},\tilde{p})  n \, ds&=& F_{1}(\tilde{v}, \tilde{p}, \tilde{\mathfrak{l}}, \tilde{\mathfrak{r}}) +\tilde{g}_1,   \\
\mathcal{J} \tilde{\mathfrak{r}}' (t) +  \int_{\partial  \mathcal{S}_0} y^\perp \cdot T(\tilde{v},\tilde{p})  n   \, ds&= &  F_{2}(\tilde{v}, \tilde{p}, \tilde{\mathfrak{l}}, \tilde{\mathfrak{r}}) + g_2,  \\
\tilde{v} |_{t= 0} = 0 & &  \text{for}  \  x\in  \mathcal{F}_0 ,   \\
 \ \tilde{\mathfrak{l}}' (0)=  0 , & &   \tilde{\mathfrak{r}}  (0)=  0, 
\end{eqnarray*}
where $ F_0, F_1$ and $ F_2 $ are defined in a similar fashion of \cite{GGH}, i.e.
\begin{gather*}
F_0(\tilde{v},\tilde{p},\tilde{\mathfrak{l}},\tilde{\mathfrak{r}}) = F_0(\tilde{v},\tilde{p}) = (\mathcal{L}-\Delta)\tilde{v}-\mathcal{M}\tilde{v}-\mathcal{G}(\tilde{p}), \quad 
F_1(\tilde{v},\tilde{p},\tilde{\mathfrak{l}},\tilde{\mathfrak{r}}) = -r(t)\tilde{\mathfrak{l}}^{\perp} +\int_{\mathcal{S}_0} (T-\mathcal{T})(\tilde{v},\tilde{p})n ds \\ \text{and} \quad F_2(\tilde{v},\tilde{p},\tilde{\mathfrak{l}},\tilde{\mathfrak{r}}) = \int_{\mathcal{S}_0} y^{\perp}\cdot(T-\mathcal{T})(\tilde{v},\tilde{p})n ds.
\end{gather*}
We consider the space
\begin{equation*}
\mathcal{K}^t = \left\{ (u,p,l,r) \in X^{t}_{4/3,4/3}\times Y^t_{4/3,4/3}\times W^{1,4/3}(0,t;\mathbb{R}^3) \quad : \quad u(0) = 0 \text{ and } (l,r)(0) = 0  \right\}
\end{equation*}
where 
\begin{gather*}
X^t_{p,q} =  W^{1,p}(0,t; L^q(\mathcal{F}_0))\cap L^p(0,t; W^{2,q}(\mathcal{F}_0)), \quad
Y^t_{p,q} =  L^p(0,t; \hat{W}^{1,q}(\mathcal{F}_0)),
\end{gather*}
and consider the map $ \phi^t $ such that 
\begin{equation*}
\phi^t(\tilde{u},\tilde{q},\tilde{l},\tilde{r} ) = (\tilde{v},\tilde{p},\tilde{\mathfrak{l}},\tilde{\mathfrak{r}})
\end{equation*}
where $ (\tilde{v},\tilde{p},\tilde{\mathfrak{l}},\tilde{\mathfrak{r}}) $ is the solution of the above system with 
\begin{equation*}
\begin{pmatrix}
F_0(\tilde{u},\tilde{q},\tilde{l},\tilde{r}) \\
F_1(\tilde{u},\tilde{q},\tilde{l},\tilde{r}) \\
F_2(\tilde{u},\tilde{q},\tilde{l},\tilde{r})
\end{pmatrix}
\end{equation*} 
in the right hand side.

It remains to show that $ \phi^t$ is a map from $ \mathcal{K}^t $ to $ \mathcal{K}^t $ and that it is a contraction if we choose $ t $ enough small. To do so we follow the estimates from \cite[Lemma 6.6 and Lemma 6.7]{GGH}. Indeed they are much easier that the ones of \cite{GGH} because the change of variables does not depend on the solution itself and no assumption on p and q is required because we do not need any embedding result.

\appendix

\section*{Appendix A.  Proof of Theorem \ref{THM:exw}}

As pointed out previously it is possible to follow the proof presented in \cite{exi:GeH} and prove that there exists a weak solution which satisfies (\ref{wek:for}) for any test function in $ \mathcal{T}_{0,T} $ (recall the definition in  \eqref{def-T}). 

We prove that the solution satisfies (\ref{wek:for}) for any test function in $\mathcal{W}_{0,T} $. The proof in \cite{exi:GeH} is based on a local-in-time existence which leads to concatenate solution up to collision, see last paragraph of Section 5.7 of \cite{exi:GeH}. Therefore it is enough to prove that the local-in-time existence result holds also for the restriction in time of the element of $\mathcal{W}_{0,T} $. We state the local-in-time existence result.
\begin{Theorem}

Let $ \Omega \subset \mathbb{R}^2 $ an open, bounded, connected set with smooth boundary, $\mathcal{S}_0$ a closed, bounded, connected and simply connected subset of $ \Omega $ with smooth boundary, $ u_0 \in\ \mathcal{H}_{S_0}$ and $ \delta > 0 $ such that $ \dist(\partial \Omega, \mathcal{S}_0)> 2\delta$. There exists $ \tau > 0 $ and a couple $(\mathcal{S},u)$ such that satisfying 
\begin{itemize}
\item $\mathcal{S}(t) \subset \Omega $ is a bounded domain of $ \mathbb{R}^2 $ for all $ t \in [0,\tau] $, such that $ \chi_S(t,x) = 1_{S(t)}(x) \in L^{\infty}((0,\tau)\times \Omega) $, 
\item $ u $ belongs to the space $ \mathcal{V}_{\tau} $ where $ \mathcal{F}(t) = \Omega \setminus \overline{\mathcal{S}(t)}$ for all $ t \in [0,\tau] $,
\item for any $ \varphi \in \mathcal{W}_{\tau} $, it holds
\begin{gather}
-\int_0^{\tau} \int_{\mathcal{F}(t)}u_{F}\cdot \partial_t \varphi_F dx dt -\int_0^{\tau}\int_{\mathcal{S}(t)}\rho_S u_S\cdot \partial_t \varphi_S dx dt -\int_0^{\tau}\int_{\mathcal{F}(t)}  u_F \otimes u_F : \nabla \varphi_F dx dt \nonumber \\
+ 2 \int_0^{\tau} \int_{\mathcal{F}(t)} Du_F :D\varphi_F dx dt + 2 \alpha \int_0^{\tau} \int_{\partial \Omega} u_F \cdot \varphi_F ds dt + 2  \alpha \int_0^{\tau} \int_{\partial \mathcal{S}(t)} (u_F-u_S)\cdot (\varphi_F -\varphi_S) ds dt  \label{wek:for2} \\
= \int_{\mathcal{F}(0)} u_{F,0}\cdot \varphi_F|_{t = 0} dx +\int_{\mathcal{S}(0)} \rho_S u_{S,0}\cdot\varphi_S|_{t=0} dx-\int_{\mathcal{F}(\tau)} u_{F}\cdot \varphi_F|_{t = \tau} dx -\int_{\mathcal{S}(\tau)} \rho_S u_{S}\cdot\varphi_S|_{t=\tau} dx.
\nonumber 
\end{gather}
\item $\mathcal{S} $ is transported by the rigid vector fields $ u_S $, i.e. for any $ \psi \in C^{\infty}([0,\tau]; \mathcal{D}(\overline{\Omega})) $, it holds
\begin{equation*}
\int_0^{\tau}\int_{\mathcal{S}(t)} \partial_t \psi + \int_0^{\tau} \int_{\mathcal{S}(t)} u_S\cdot \nabla \psi =\int_{\mathcal{S}(\tau)}\psi|_{t=\tau} -\int_{\mathcal{S}_0}\psi|_{t=0}.
\end{equation*}
\item $ \dist(\partial \Omega, \mathcal{S}(t)) \geq 2 \delta $ for almost any $ t \in [0,\tau] $.
\end{itemize}

\end{Theorem}

\begin{proof}
By the proof in \cite{exi:GeH} we already know that this theorem holds with test functions in $ \mathcal{T}_{\tau} $, which is the set of $\varphi|_{[0,\tau]} $ where $ \varphi \in \mathcal{T}_{0,T}$.
We prove that (\ref{wek:for2}) holds for any test function in $ \mathcal{W}_{\tau} $. To do so we approximate the test functions in $ \mathcal{W}_{\tau} $ by admissible test functions of the approximate problem defined in  \cite[Section 2]{exi:GeH}.
To do so we need an equivalent of Proposition 12 in \cite{exi:GeH}, i.e. we prove the following claim.

\begin{Claim}
Let $ \alpha > 0 $ and let $ \varphi \in \mathcal{W}_{\tau} $. Then there exists a sequence $ \varphi^n \in W^{1,\infty}(0,\tau;L^2_{\sigma}(\Omega)) \cap L^{\infty}(0,\tau;H^{1}_{\sigma}(\Omega)) $ of the form 
\begin{equation*}
\varphi^n = (1-\chi^n_S)\varphi_F + \chi^n_S \varphi^n_S,
\end{equation*}
that satisfies 
\begin{itemize}
\item $ \| \sqrt{\chi^n_S}(\varphi^n_S-\varphi_S)\|_{C([0,\tau];L^p(\Omega))} = O(n^{-\alpha/(p+\varepsilon)}) \text{ for any } p \in (2, \infty) \text{ and for any } \varepsilon > 0 $,

\item $ \varphi^n \to \varphi $ \text{ strongly in } $ C([0,\tau]; L^{p}) $ \text{ for } $ p > 2 $, 

\item $ \| \varphi^n\|_{C([0,\tau];H^1(\Omega))} = O(n^{\alpha(1-1/p)})$ \text{ for any } $ p > 1 $,

\item $ \|\chi^n_s(\partial_t+P^n_S u^n\cdot \nabla)(\varphi^n_S-\varphi_S)\|_{L^{\infty}(0,\tau;L^p(\Omega))} = O(n^{-\alpha/(p+\varepsilon)}) \text{ for any } p \in (2, \infty) \text{ and for any } \varepsilon > 0 $, 

\item $ (\partial_t +P^n_S u^n\cdot\nabla)\varphi^n \to (\partial_t + P_S u \cdot \nabla) \varphi $ \text{ weakly* in } $L^{\infty}(0,\tau; L^p(\Omega)) $.
\end{itemize}
\end{Claim}

\begin{proof}[Proof of the claim.] To prove this claim we make the same construction than \cite{exi:GeH}. The main difficulty is not the lack of regularity in time but the lack of regularity in space. By the fact that the construction of this approximation is quite technical and involved, we present here quite rapidly the construction and we refer to \cite[Section 5.3]{exi:GeH} for more details. 

Recall that a weak solution $(\mathcal{S}, u) $, constructed in \cite{exi:GeH} comes as a limit of solutions $ (\mathcal{S}_n, u_n) $ of some approximate problems and recall $ \phi_n \in W^{1,\infty}(0,\tau; C^{\infty}) $ the flow from \cite{exi:GeH}, i.e. $ \phi_n(t): \mathbb{R}^2 \to \mathbb{R}^2 $ is a $ C^{\infty }$-diffeomorphism and it is the flow associated with $ P^n_S u_n $.
We define an approximation $ \varphi_n $ of $ \varphi \in \mathcal{W}_{\tau}$ using the flow $ \phi_n $. 

The idea of G\'erard-Varet and Hillairet is to use $ \phi_n $ to translate the problem to a ``fixed'' domain and then approximate. Define $ \tilde{\Phi}^n_S $ and $ \tilde{\Phi}^n_F $ via
\begin{equation*}
\varphi_S(t, \phi_n(t,y)) = d\phi_n(t,y)(\tilde{\Phi}^n_S(t,y)), \quad \varphi_F(t, \phi_n(t,y)) = d\phi_n(t,y)(\tilde{\Phi}_F^n(t,y)).   
\end{equation*}     
$ \tilde{\Phi}^n_S $ and $ \tilde{\Phi}^n_F $ are defined in a fixed solid domain in the sense that the solid part is fixed, i.e. $ \phi_n|_{\mathcal{S}_0} : \mathcal{S}_0 \to \mathcal{S}_n(t) $. In the approximation we do not change $ \varphi $ in the fluid part so we define $ \varphi_n|_{(0,\tau)\times \mathcal{F}^n(t)} = \varphi|_{(0,\tau)\times \mathcal{F}_n(t)} $ and $ \varphi_n $ in the solid part such that it is closed to a solid rotation and such that it makes $ \varphi_n $ an $ L^2(0,\tau;H^1_{\sigma}(\Omega)) $ function. To do so we approximate $ \tilde{\Phi}^n_S $ by $ \Phi^n_S = \Phi^n_{1,S} + \Phi^n_{2,S} $, where 
\begin{equation}
\label{con:phi}
\Phi^n_{1,S} = \tilde{\Phi}^n_S+\chi(n^{\alpha}z)((\tilde{\Phi}^n_F -\tilde{\Phi}^n_S)-[(\tilde{\Phi}^n_S-\tilde{\Phi}^n_F)\cdot e_z]e_z ),
\end{equation}
and $ \Phi_{2,S}^n $ is defined in such a way to make $ \Phi_S^n $ divergence free. These lead us to define 
\begin{equation*}
\varphi^{n}_S(t,\phi_n(t,y)) = d\phi_n(t,y) (\Phi^n_S(t,y)).
\end{equation*}
To conclude the proof of the claim we have to present the estimates. 
\begin{align*}
\| \Phi^n_{1,S}-\Phi_S \|_{W^{1,\infty}(0,\tau; L^p(\mathcal{S}_0)) } = & \| \chi(n^{\alpha}z)((\tilde{\Phi}^n_F -\tilde{\Phi}^n_S)-[(\tilde{\Phi}^n_S-\tilde{\Phi}^n_F)\cdot e_z]e_z )\|_{W^{1,\infty}(0,\tau; L^p(\mathcal{S}_0)))} \\
\leq & \|\chi(n^{\alpha}z)\|_{L^q{(\mathcal{S}_0))}}\| (\tilde{\Phi}^n_F -\tilde{\Phi}^n_S)-[(\tilde{\Phi}^n_S-\tilde{\Phi}^n_F)\cdot e_z]e_z \|_{W^{1,\infty}(0,\tau; L^r(\mathcal{S}_0)))} \\ 
\leq & C(r)\|\chi(n^{\alpha}z)\|_{L^q{(\mathcal{S}_0))}}\| (\tilde{\Phi}^n_F -\tilde{\Phi}^n_S)-[(\tilde{\Phi}^n_S-\tilde{\Phi}^n_F)\cdot e_z]e_z \|_{W^{1,\infty}(0,\tau; H^1(\mathcal{S}_0)))} \\
\leq & C(r)n^{-\alpha/q}.
\end{align*}
where $ 1/q+1/r = 1/p $. In similar way
\begin{align*}
\| \Phi^n_{1,S}-\Phi_S \|_{W^{1,\infty}(0,\tau; H^1(\mathcal{S}_0)) } = & \| \chi(n^{\alpha}z)((\tilde{\Phi}^n_F -\tilde{\Phi}^n_S)-[(\tilde{\Phi}^n_S-\tilde{\Phi}^n_F)\cdot e_z]e_z )\|_{W^{1,\infty}(0,\tau; H^1(\mathcal{S}_0)))} \\
\leq & C + \|n^{\alpha}\nabla\chi(n^{\alpha}z)\|_{L^q{(\mathcal{S}_0))}}\| (\tilde{\Phi}^n_F -\tilde{\Phi}^n_S)-[(\tilde{\Phi}^n_S-\tilde{\Phi}^n_F)\cdot e_z]e_z \|_{W^{1,\infty}(0,\tau; L^r(\mathcal{S}_0)))} \\ 
\leq & C + C(r)n^{\alpha(1-1/q)}\| (\tilde{\Phi}^n_F -\tilde{\Phi}^n_S)-[(\tilde{\Phi}^n_S-\tilde{\Phi}^n_F)\cdot e_z]e_z \|_{W^{1,\infty}(0,\tau; H^1(\mathcal{S}_0)))} \\
\leq & C + C(r)n^{\alpha(1-1/q)}.
\end{align*}
Moreover
\begin{align*}
\| \Phi^n_{2,S} \|_{W^{1,\infty}(0,\tau; L^p(\mathcal{S}_0)) } \leq & C(r)\| \Phi^n_{2,S} \|_{W^{1,\infty}(0,\tau; W^{1,r}(\mathcal{S}_0)) }\\
\leq & C(r) \| \div \Phi^n_{1,S} \|_{W^{1,\infty}(0,\tau; L^r(\mathcal{S}_0)) } \\
= & C(r) \| \chi(n^{\alpha}z)\mathcal{J}^n\|_{W^{1,\infty}(0,\tau; L^r(\mathcal{S}_0)) } \\
\leq & C(r) \| \chi(n^{\alpha}z)\|_{L^p(\mathcal{S}_0)}\|\mathcal{J}^n\|_{W^{1,\infty}(0,\tau; L^2(\mathcal{S}_0)) }  \\
\leq & C(r) n^{-\alpha/p}.
\end{align*}
where $ \mathcal{J}^n$ is defined in \cite{exi:GeH}, $ 1-2/r = - 2/p $ i.e. $ r = 2p/(p+2)$ and it holds $ 1/p + 1/2 = 1/r $. In a similar way
\begin{align*}
\| \Phi^n_{2,S} \|_{W^{1,\infty}(0,\tau;H^1(\mathcal{S}_0)) } \leq  \| \div \Phi^n_{1,S} \|_{W^{1,\infty}(0,\tau; L^2(\mathcal{S}_0)) }
\leq  C.
\end{align*}

The estimates above prove the first three points of the claim. For the last two points we follow the computation of \cite{exi:GeH}, namely
\begin{align*}
\left\|\chi^n_S(\partial_t + P^n_S u^n\cdot \nabla)(\varphi^n_S-\varphi_S))\right\|_{L^{\infty}(0,\tau;L^{p}(\Omega))} \leq \left\| \frac{\partial}{\partial t} d\phi^n(t,y)(\Phi^n_S-\Phi_S)\right\|_{L^{\infty}(0,\tau;L^{p}(\Omega))} \leq  C n^{-\alpha/(p+\varepsilon)}.
\end{align*} 
For the last point we compute
\begin{align*}
(\partial_t+P^n_Su^n\cdot \nabla) \varphi^n -(\partial_t+P_Su\cdot \nabla) \varphi = & (1-\chi^n_S)(\partial_t+P^n_Su^n\cdot \nabla)(\varphi^n_F-\varphi_F)+\chi^n_S(\partial_t+P^n_Su^n\cdot \nabla)(\varphi^n_S-\varphi_S) \\ & +(1-\chi^n_S)(\partial_t+P^n_Su^n\cdot \nabla)\varphi_F+\chi^n_S(\partial_t+P^n_Su^n\cdot \nabla)\varphi_S,
\end{align*}
which converge converge weakly* to $ 0 $ by the strong convergence of $ \chi^n_S $ and the weak convergence of $ P^n_S u^n$.
\end{proof}

The above claim prove that there exists a good approximation $ \varphi_n $, for $ \varphi \in \mathcal{W}_{\tau} $ that leads us to pass to the limit in the approximate problem. This means that we can test the weak formulation with any function in $ \mathcal{W}_{\tau}$.

\end{proof}

\ \par \ 

 {\bf Acknowledgements.} 
The author was supported by the Agence Nationale de la Recherche, 
 Project IFSMACS, grant ANR-15-CE40-0010 and the Conseil R\'egional dÄ'Aquitaine, grant 2015.1047.CP.


\begin{thebibliography}{50}

\bibitem{STR} Baba, H. A., Amrouche, C., Escobedo, M. (2017). Maximal $ L^ p $-$ L^ q $ regularity for the Stokes problem with Navier-type boundary conditions. arXiv preprint arXiv:1703.06679. 

\bibitem{STR2} Baba, H. A., Chemetov, N. V., Ne\v{c}asov\'a, \v{S}.,  Muha, B. (2017). Strong solutions in $ L^2$ framework for fluid-rigid body interaction problem-mixed case. arXiv preprint arXiv:1707.00858. 

\bibitem{BFNW} Bucur, D., Feireisl, E., Ne\v{c}asov\'a, \v{S}., Wolf, J. (2008). On the asymptotic limit of the Navier-€"Stokes system on domains with rough boundaries. Journal of Differential Equations, 244(11), 2890-2908.

\bibitem{exi:wea} Chemetov, N. V., Ne\v{c}asov\'a, \v{S}. (2017). The motion of the rigid body in the viscous fluid including collisions. Global solvability result. Nonlinear Analysis: Real World Applications, 34, 416-445.

\bibitem{wea:str} Chemetov, N. V., Ne\v{c}asov\'a, \v{S}.,  Muha, B. (2017). Weak-strong uniqueness for fluid-rigid body interaction problem with slip boundary condition. arXiv preprint arXiv:1710.01382.

\bibitem{cha:tak} Cumsille, P., Takahashi, T. (2008). Well-posedness for the system modelling the motion of a rigid body of arbitrary form in an incompressible viscous fluid. Czechoslovak Mathematical Journal, 58(4), 961-992.

\bibitem{GGH} Geissert, M., G\"otze, K., Hieber, M. (2013). $L^p$-theory for strong solutions to fluid-rigid body interaction in Newtonian and generalized Newtonian fluids. Transactions of the American Mathematical Society, 365(3), 1393-1439.

\bibitem{exi:GeH} G\'erard-€Varet, D., Hillairet, M. (2014). Existence of Weak Solutions Up to Collision for Viscous Fluid-Solid Systems with Slip. Communications on Pure and Applied Mathematics, 67(12), 2022-2076.

\bibitem{DGHW} G\'erard-Varet, D., Hillairet, M., Wang, C. (2015). The influence of boundary conditions on the contact problem in a 3D Navier-€"Stokes flow. Journal de Math\'ematiques Pures et Appliqu\'ees, 103(1), 1-38.

\bibitem{UnGS} Glass, O., Sueur, F. (2015). Uniqueness results for weak solutions of two-dimensional fluid-solid systems. Archive for Rational Mechanics and Analysis, 218(2), 907-944.

\bibitem{GLS} Gunzburger, M. D., Lee, H. C., Seregin, G. A. (2000). Global existence of weak solutions for viscous incompressible flows around a moving rigid body in three dimensions. Journal of Mathematical Fluid Mechanics, 2(3), 219-266.

\bibitem{H} Hillairet, M. (2007). Lack of collision between solid bodies in a 2D incompressible viscous flow. Communications in Partial Differential Equations, 32(9), 1345-1371.

\bibitem{IW} Inoue, A.,  Wakimoto, M. (1977). On existence of solutions of the Navier-Stokes equation in a time dependent domain. J. Fac. Sci. Univ. Tokyo Sect. IA Math, 24(2), 303-319.

\bibitem{J}  Kolumban, J. J. (2018). Control at a distance of the motion of a rigid body immersed in a two-dimensional viscous incompressible. In preparation. 

\bibitem{str:dem} Maity, D., Tucsnak, M. (2017). $ L^ p $-$ L^ q $ Maximal Regularity for some Operators Associated with Linearized Incompressible Fluid-Rigid Body Problems. arXiv preprint arXiv:1712.00223.

\bibitem{pla:sue} Planas, G., Sueur, F. (2014). On the ``viscous incompressible fluid+ rigid body'' system with Navier conditions. In Annales de l'Institut Henri Poincar\'e (C) Non Linear Analysis (Vol. 31, No. 1, pp. 55-80). Elsevier Masson. 

\bibitem{SMST} San Marti­n, J. A., Starovoitov, V., Tucsnak, M. (2002). Global Weak Solutions for the Two-Dimensional Motion of Several Rigid Bodies in an Incompressible Viscous Fluid. Archive for Rational Mechanics and analysis, 161(2), 113-147.

\bibitem{shimada} Shimada, R. (2007). On the $L^p$-$L^q$ maximal regularity for Stokes equations with Robin boundary condition in a bounded domain. Mathematical methods in the applied sciences, 30(3), 257-289.

\bibitem{WAN} Wang, C. (2014). Strong solutions for the fluid-€"solid systems in a 2-D domain. Asymptotic Analysis, 89(3-4), 263-306. 

\end{thebibliography}
\end{document}